\documentclass[11pt]{article}
\usepackage{enumerate}
\usepackage{amsmath}
\usepackage{amsthm}
\usepackage{amsfonts}
\usepackage{amssymb}
\usepackage[numbers]{natbib}
\usepackage{xcolor}
\usepackage{graphicx}
\def\M{\mathcal{M}}
\def\Pr{\mathbb{P}}
\def\eps{\varepsilon}

\usepackage{dsfont} 

\usepackage{fullpage}

\def\COMMENT#1{}
\let\COMMENT=\footnote

\usepackage[capitalise]{cleveref} 
\crefname{equation}{}{}
\crefname{enumi}{}{}

\usepackage{tikz}
\usepackage{calc}
\usetikzlibrary{arrows,shapes}
\usetikzlibrary{calc}
\usetikzlibrary{math}
\usetikzlibrary{decorations.pathreplacing}
\tikzstyle{vertex}=[circle,fill=black,minimum size=5pt,inner sep=0pt]
\tikzstyle{edge} = [draw,thick,-]
\newcommand\vara{3}
\newcommand\varb{5.5}

\def\vp#1{}
\renewcommand{\vp}[1]{\footnote{\textcolor{green!40!black}{\textbf{VP: }#1}}} 
\newcommand{\s}[1]{\left\lvert #1 \right\rvert}
\usepackage{mathtools}
\usepackage{enumitem}

\newcommand{\keywords}[1]
{
	{\small\textbf{Keywords:} #1}
}

\newcommand{\prob}[1]{\mathbb{P}\left[ #1 \right]}
\newcommand{\Gm}{\mathbf{G}_{\mu}}
\newcommand{\Hm}{\mathbf{H}_{\mu}}

\newtheorem{question}{Question}
\newtheorem{corollary}[question]{Corollary}
\newtheorem{problem}[question]{Problem}
\newtheorem{conjecture}[question]{Conjecture}
\newtheorem{theorem}[question]{Theorem}

\newtheorem{proposition}[question]{Proposition}
\newtheorem{lemma}[question]{Lemma}
\newtheorem{remark}[question]{Remark}

\newtheorem{definition}[question]{Definition}

\numberwithin{question}{section}
\numberwithin{equation}{section}

\title{$1$-independent percolation on $\mathbb{Z}^2\times K_n$}
\author{Victor Falgas-Ravry\thanks{Ume{\aa} Universitet, Sweden. Email: \texttt{victor.falgas-ravry@umu.se}. Research supported by Swedish Research Council grants VR 2016-03488 and VR 2021-03687.} \and Vincent Pfenninger\thanks{University of Birmingham, UK. Email: \texttt{vxp881@bham.ac.uk}.}}
\begin{document}
	\maketitle	
\begin{abstract}
A random graph model on a host graph~$H$ is said to be \emph{$1$-independent} if for every pair of vertex-disjoint subsets~$A,B$ of~$E(H)$, the state of edges (absent or present) in~$A$ is independent of the state of edges in~$B$.  For an infinite connected graph~$H$, the \emph{$1$-independent critical percolation probability}~$p_{1,c}(H)$ is the infimum of the $p\in [0,1]$ such that every $1$-independent random graph model on~$H$ in which each edge is present with probability at least~$p$ almost surely contains an infinite connected component.

Balister and Bollob\'as observed in 2012 that $p_{1,c}(\mathbb{Z}^d)$ is nonincreasing and tends to a limit in~$[\frac{1}{2}, 1]$ as $d\rightarrow \infty$. They asked for the value of this limit. We make progress towards this question by showing that \[\lim_{n\rightarrow \infty}p_{1,c}(\mathbb{Z}^2\times K_n)=4-2\sqrt{3}=0.5358\ldots \ .\]  In fact, we show that the equality above remains true if the sequence of complete graphs~$K_n$ is replaced by a sequence of weakly pseudorandom graphs on~$n$ vertices with average degree $\omega(\log n)$. We conjecture that the equality also remains true if~$K_n$ is replaced instead by the $n$-dimensional hypercube~$Q_n$. This latter conjecture would imply the answer to Balister and Bollob\'as's question is $4-2\sqrt{3}$.

Using our results, we are also able to resolve a problem of Day, Hancock and the first author on the emergence of long paths in $1$-independent random graph models on $\mathbb{Z}\times K_n$. Finally, we prove some results on component evolution in $1$-independent random graphs, and discuss a number of open problems arising from our work that may pave the way for further progress on the question of Balister and Bollob\'as.
\end{abstract}	

\keywords{percolation theory, extremal graph theory, locally dependent random graphs}
	
\section{Introduction}	
\subsection{Background}
Percolation theory lies at the interface of probability theory, statistical physics and combinatorics. Its object of study is, roughly speaking, the connectivity properties of random subgraphs of infinite connected graphs, and in particular the points at which these undergo drastic transitions such as the emergence of infinite components. Since its inception in Oxford in the late 1950s, percolation theory has become a rich field of study (see e.g.\ the monographs~\cite{BollobasRiordan06,Grimmett99,MeesterRoy96}).  One of the cornerstones of the discipline is the Harris--Kesten Theorem~\cite{Harris60,Kesten80}, which states that if each edge of the integer square lattice $\mathbb{Z}^2$ is open independently at random with probability~$p$, then if $p\leq \frac{1}{2}$ almost surely all connected components of open edges are finite, while if $p>\frac{1}{2}$ almost surely there exists an infinite connected component of open edges. Thus~$1/2$ is what is known as the \emph{critical probability for independent bond percolation} on $\mathbb{Z}^2$.

In general, given an infinite connected graph~$H$, determining the critical probability for independent bond percolation on~$H$ is a hard problem, with the answer known exactly only in a handful of cases. There is thus great interest in methods for rigorously estimating such critical probabilities. One of the most powerful and effective techniques for doing just that was developed by Balister, Bollob\'as and Walters~\cite{BalisterBollobasWalters05}, and relies on comparing percolation processes with locally dependent  bond percolation on $\mathbb{Z}^2$ (to be more precise: $1$-independent bond percolation; see below for a definition). The method of Balister, Bollob\'as and Walters has proved influential, and has been widely applied to obtain the best rigorous confidence interval estimates for the value of the critical parameter in a wide range of models, see e.g.~\cite{BalintBeffaraTassion13,BalisterBollobas13,BalisterBollobasWalters05, BalisterBollobasSarkarKumar07,BalisterBollobasWalters09, Ball14, BenjaminiStauffer13, DeijfenHaggstromHolroyd12, DeijfenHolroydPeres11, HaenggiSarkar13, RiordanWalters07}.

However, as noted by the authors of~\cite{BalisterBollobasWalters05} and again by Balister and Bollob\'as~\cite{BalisterBollobas12} in 2012, locally dependent bond percolation is poorly understood. To quote from the latter work, ``\emph{[given that] $1$-independent percolation models have become a key tool in establishing bounds on critical probabilities [...], it is perhaps surprising that some of the most basic questions about $1$-independent models are open}''. In particular, there is no known locally dependent analogue of the Harris--Kesten Theorem, nor even until now much of a sense of what the corresponding $1$-independent critical probability ought to  be. In this paper, we contribute to the broader project initiated by Balister and Bollob\'as of addressing the gap in our knowledge about $1$-independent bond percolation by making some first steps towards a $1$-independent Harris--Kesten Theorem. To state our results and place them in their proper context, we first need to give some definitions.

Let $H=(V,E)$ be a graph. Given a probability measure~$\mu$ on subsets of~$E$, a \emph{$\mu$-random graph} $\mathbf{H}_\mu$ is a random spanning subgraph of~$H$ whose edge-set is chosen randomly from subsets of~$E$ according to the law given by~$\mu$. Each probability measure~$\mu$ on subsets of~$E$ thus gives rise to a \emph{random graph model} on the host graph~$H$, and we use the two terms (probability measure~$\mu$ on subsets of~$E$/random graph model $\mathbf{H}_{\mu}$ on~$H$) interchangeably. In this paper we will be interested in random graph models where the state (present/absent) of edges is dependent only on the states of nearby edges. Recall that the \emph{graph distance} between two subsets $A,B\subseteq E$ is  the length of the shortest path in~$H$ from an endpoint of an edge in~$A$ to an endpoint of an edge in~$B$. So in particular if an edge in~$A$ shares a vertex with an edge in~$B$, then the graph distance from~$A$ to~$B$ is zero, while if~$A$ and~$B$ are supported on disjoint vertex-sets, then the graph distance from~$A$ to~$B$ is at least one.
\begin{definition}[$k$-independence]
	A random graph model $\mathbf{H}_{\mu}$ on  a host graph~$H$ is \emph{$k$-independent} if whenever $A, B$ are disjoint subsets of~$E(H)$ such that the graph distance between~$A$ and~$B$ is at least~$k$, the random variables $E(\mathbf{H}_{\mu})\cap A$ and $E(\mathbf{H}_{\mu})\cap B$ are mutually independent. If $\mathbf{H}_{\mu}$ is $k$-independent, we say that the associated probability measure~$\mu$ is a \emph{$k$-independent measure}, or \emph{$k$-ipm}, on~$H$.
\end{definition} 
Let $\mathcal{M}_{k,\geq p}(H)$ denote the collection of all $k$-independent measures~$\mu$ on~$E(H)$ in which each edge of~$H$ is included in  $\mathbf{H}_{\mu}$ with probability at least~$p$. We define $\mathcal{M}_{k,\leq p}(H)$ mutatis mutandis, and let $\mathcal{M}_{1,p}(H)$ denote $\mathcal{M}_{k, \geq p}\cap\mathcal{M}_{k, \leq p}$ --- in other words $\mathcal{M}_{k,p}$ is the collection of all $k$-ipm~$\mu$ on~$H$  in which each edge of~$H$ is included in  $\mathbf{H}_{\mu}$ with probability exactly~$p$.

Observe that a $0$-independent measure~$\mu$ is what is known as a \emph{Bernoulli} or \emph{product} measure on~$E$: each edge in~$E$ is included in~$\mathbf{H}_{\mu}$ at random independently of all the others. We refer to such measures as \emph{independent measures}. The collection $\mathcal{M}_{0,p}(H)$ thus consists of a single measure, the \emph{$p$-random measure}, in  which each edge of~$H$ is included in the associated random graph with probability~$p$, independently of all the other edges. When the host graph~$H$ is~$K_n$, the complete graph on~$n$ vertices, this gives rise to the celebrated \emph{Erd{\H o}s--R\'enyi random graph model}, while when $H=\mathbb{Z}^2$ this is exactly the \emph{independent bond percolation} model considered in the Harris--Kesten Theorem.

In this paper, we will focus instead on $\mathcal{M}_{1,\geq p}(H)$ and $\mathcal{M}_{1, p}(H)$, whose probability measures allow for some local dependence between the edges.  A simple and well-studied example of a model from $\mathcal{M}_{1,p}(H)$ is given by \emph{site percolation}: build a random spanning subgraph $\mathbf{H}^{\mathrm{site}}_{\theta}$ of~$H$ by assigning each vertex $v\in V(H)$ a state~$S_v$ independently at random, with $S_v=1$ with probability~$\theta$ and $S_v=0$ otherwise, and including an edge $uv\in E(H)$ in $\mathbf{H}^{\mathrm{site}}_{\theta}$ if and only if $S_u=S_v=1$. Each edge in this random graph is open with probability $p=\theta^2$, and the model is clearly $1$-independent since `randomness resides in the vertices', and so what happens inside two disjoint vertex sets is independent. More generally, any \emph{state-based model} obtained by first assigning independent random states~$S_v$ to vertices $v\in V(H)$ and then adding an edge~$uv$ according to some deterministic or probabilistic rule depending only on the ordered pair $(S_u, S_v)$ will give rise to a $1$-ipm on~$H$. State-based models are a generalisation of the probabilistic notion of a \emph{two-block factor}, see~\cite{LiggettSchonmannStacey97} for details.

Given a $1$-ipm~$\mu$ on an infinite connected graph~$H$, we say that~$\mu$ \emph{percolates} if $\mathbf{H}_{\mu}$ almost surely (i.e.\ with probability~$1$) contains an infinite connected component.\footnote{Note the existence of an infinite connected component is a tail event, in the sense that one cannot create or destroy an infinite connected component by changing the state of finitely many edges, so that by a $1$-independent version of Kolmogorov's zero--one law, $\mathbf{H}_{\mu}$ contains an infinite connected component with probability $0$ or $1$ (see the discussion below Theorem 1 in~\cite[Chapter 2]{BollobasRiordan06}).}
\begin{definition}
Given an infinite connected graph~$H$, we define the \emph{$1$-independent critical percolation probability} for~$H$ to be 
\[p_{1,c}(H):=\inf\left\{p\geq 0: \ \forall \mu \in \mathcal{M}_{1,\geq p}(H), \ \mu \emph{ percolates} \right\}.\]
\end{definition}
\begin{remark}
Given $\mu \in \mathcal{M}_{1,\geq p}(H)$ we can obtain a random graph $\mathbf{H}_{\nu}$ from $\mathbf{H}_{\mu}$ by deleting each edge~$uv$ of $\mathbf{H}_{\mu}$ independently at random with probability $1-p/\left(\Pr[uv \in E(\mathbf{H}_{\mu})]\right)$. Clearly $\mathbf{H}_{\mu}$  stochastically dominates (i.e.\ is a supergraph of) $\mathbf{H}_{\nu}$ and $\nu \in \mathcal{M}_{1,p}(H)$. Thus the definition of~$p_{1,c}(H)$ above is unchanged if we replace $\mathcal{M}_{1,\geq p}(H)$ by $\mathcal{M}_{1,p}(H)$.
\end{remark}
\begin{remark}
The probability~$p_{1,c}(H)$ is in fact one of five natural critical probabilities for $1$-independent percolation one could consider, all of which are distinct  in general --- see~\cite[Section 11.3, Corollary 50 and Question 53]{DayFalgasRavryHancock20}.
\end{remark}

Balister, Bollob\'as and Walters~\cite{BalisterBollobasWalters05} devised a highly effective method for giving rigorous confidence interval results for critical parameters in percolation theory via comparison with $1$-independent models on the square integer lattice $\mathbb{Z}^2$. Their method relies on estimating the probability of certain \emph{finite, bounded} events (usually via Monte Carlo methods, whence the confidence intervals) and on bounds on the $1$-independent critical probability $p_{1,c}(\mathbb{Z}^2)$. Work of Liggett, Schonman and Stacey~\cite{LiggettSchonmannStacey97} on stochastic domination of independent models by $1$-independent models implied $p_{1,c}(\mathbb{Z}^2)<1$. Balister, Bollob\'as and Walters~\cite[Theorem 2]{BalisterBollobasWalters05} obtained the effective upper bound $p_{1,c}(\mathbb{Z}^2)<0.8639$ via a renormalisation argument; this upper bound has not been improved since, and the authors of~\cite{BalisterBollobasWalters05} noted \emph{``it would be of interest to give significantly better bounds for $p_{1,c}(\mathbb{Z}^2)$; unfortunately, we cannot even hazard a guess as to [its] value''}. The question of determining~$p_{1,c}(\mathbb{Z}^2)$ was raised again by Balister and Bollob\'as~\cite[Question 2]{BalisterBollobas12}, who noted the difficulty of the problem:
\begin{problem}[$1$-independent Harris--Kesten problem]\label{problem: HK}
Determine $p_{1,c}(\mathbb{Z}^2)$.
\end{problem}
\noindent Balister and Bollob\'as~\cite{BalisterBollobas12} observed that a simple modification of site percolation due to to Newman shows that $p_{1,c}(\mathbb{Z}^2) \geq (\theta_{s})^2 + (1-\theta_s)^2$, where $\theta_s=\theta_s(\mathbb{Z}^2)$ is the critical probability for site percolation in $\mathbb{Z}^2$.  Since it is known that $\theta_{s}\in [0.556, 0.679492]$ (see~\cite{VDBErmakov96,Wierman95}), this shows that $p_{1,c}(\mathbb{Z}^2) \geq 0.5062$. Non-rigorous simulation-based estimates $\theta_{s}\approx 0.597246$~\cite{Ziff92} improve this to a non-rigorous lower bound of~$0.5172$.  Recently, Day, Hancock and the first author gave significant improvements on these lower bounds. In~\cite[Theorem 7]{DayFalgasRavryHancock20}, they constructed measures based on an idea from the first author's PhD thesis~\cite[Theorem 62]{FalgasRavry12} showing that for any $d\in \mathbb{N}$, $p_{1,c}(\mathbb{Z}^d)\geq 4-2\sqrt{3}=0.5358\ldots $. They in fact showed $p_{1,c}(H)\geq 4-2\sqrt{3}$ for any host graph~$H$ satisfying what they call the \emph{finite $2$-percolation property} (see Section~\ref{section: proofs of main theorems} for a formal definition), a family which includes the graphs $\mathbb{Z}^2\times K_n$ for any $n\in \mathbb{N}$. (Recall that the Cartesian product $H\times K_n$ of a graph $H$ with $K_n$ is the graph whose vertices are the pairs $(v, i)\in V(H)\times \{1,2,\ldots n\}$ and in which two distinct vertices $(v,i)$ and $(v',i')$ are joined by an edge if either $v=v'$ or $vv'$ is an edge of $H$ and $i=i'$; see Section~\ref{section: notation} for an illustration and a more general definition of the Cartesian product of two graphs.)
Further, the same authors gave a different construction~\cite[Theorem 8]{DayFalgasRavryHancock20} showing that 
\begin{align}\label{eq: lower bound on crit prob from modified site percolation}
p_{1,c}(\mathbb{Z}^2)\geq (\theta_s)^2+\frac{1-\theta_s}{2},
\end{align}
where $\theta_s=\theta_s(\mathbb{Z}^2)$ is the critical probability for site percolation in $\mathbb{Z}^2$. Using the aforementioned simulation-based estimates for~$\theta_s$, this gives a non-rigorous lower bound of~$0.5549$ on $p_{1,c}(\mathbb{Z}^2)$. All these lower bounds remain far apart from the upper bound of~$0.8639$ from~\cite{BalisterBollobasWalters05}, and, as noted in~\cite{BalisterBollobasWalters05}, part of the difficulty of Problem~\ref{problem: HK} has been the absence of a clear candidate conjecture to aim for.

In  view of the difficulty of Problem~\ref{problem: HK}, there has been interest in increasing our understanding of $1$-independent models on other host graphs than $\mathbb{Z}^2$. Balister and Bollob\'as noted $p_{1,c}(\mathbb{Z}^d)$ is non-increasing in~$d$ and must therefore converge to a limit as $d\rightarrow \infty$. They showed this limit is at least~$1/2$ and posed the following problem~\cite[Question 2]{BalisterBollobas12}:
\begin{problem}[Balister and Bollob\'as problem]\label{problem: Balister Bollobas}
	Determine $\lim_{d\rightarrow \infty} p_{1,c}(\mathbb{Z}^d)$.	
\end{problem}
\noindent By the construction of Day, Falgas-Ravry and Hancock mentioned above, this limit is in fact at least $4-2\sqrt{3}$; the only known upper bound is again the~$0.8639$ upper bound on $p_{1,c}(\mathbb{Z}^2)$ from~\cite{BalisterBollobasWalters05}.

Balister and Bollob\'as have further studied $1$-independent models on infinite trees, obtaining in this setting $1$-independent analogues of classical results of Lyons~\cite{Lyons90} for independent bond percolation. Day, Hancock and the first author for their part gave a number of results on the connectivity of $1$-independent random graphs on paths and complete graphs, and on the almost sure emergence of arbitrarily long paths in $1$-independent models.  More precisely, they introduced the \emph{Long Paths critical probability} $p_{1,LP}(H)$ of~$H$, given by
\[p_{1, LP}(H):=\inf \left\{p\in [0,1]: \ \forall \mu \in \mathcal{M}_{1,p}, \forall \ell \in \mathbb{N}, \  \mathbb{P}\left[\mathbf{H}_{\mu} \textrm{ contains a path of length }\ell\right]>0  \right\},\]
and showed $p_{1, LP}(\mathbb{Z})=3/4$, $p_{1, LP}(\mathbb{Z}\times K_2)=2/3$. Since the sequence $p_{1, LP}(\mathbb{Z}\times K_n)$ is non-increasing in~$n$, it tends to a limit in~$[0,1]$ as $n\rightarrow\infty$. Day, Hancock and the first author showed in~\cite[Theorem 12(v)]{DayFalgasRavryHancock20} that this limit lies in the interval $[4-2\sqrt{3}, 5/9]$ and asked~\cite[Problem 54]{DayFalgasRavryHancock20}:
\begin{problem}[Day, Falgas--Ravry and Hancock]\label{problem: Day FR Hancock}
Determine $\lim_{n\rightarrow\infty }p_{1, LP}(\mathbb{Z}\times K_n)$.
\end{problem}

\subsection{Contributions of this paper}
Our main result in this paper is determining the limit of the $1$-independent critical probability for percolation in $\mathbb{Z}^2\times K_n$ as $n\rightarrow \infty$:
	\begin{theorem}\label{theorem: 1-indep percolation in Z^2 times Kn} The following hold:
		\begin{enumerate}[label = \upshape{(\roman*)}]
\item if $p>4-2\sqrt{3}$ is fixed, then there exists $N\in \mathbb{N}$ such that $p_{1,c}\left(\mathbb{Z}^2\times K_N\right)\leq p$;
\item for every $n\in \mathbb{N}$, $p_{1,c}\left(\mathbb{Z}^2\times K_n\right)\geq 4-2\sqrt{3}$.
\end{enumerate}
In particular, we have
$\lim_{n\rightarrow \infty}p_{1,c}(\mathbb{Z}^2\times K_n)=4-2\sqrt{3}=0.5358\ldots$ .
\end{theorem}
\noindent As a corollary to the key result in our proof of Theorem~\ref{theorem: 1-indep percolation in Z^2 times Kn}, we also obtain a solution to the problem of Day, Falgas--Ravry and Hancock on long paths in $1$-independent percolation, Problem~\ref{problem: Day FR Hancock} above:
\begin{theorem}\label{theorem: answer to Day FR Hancock question}
	$\lim_{n\rightarrow \infty}p_{1, LP}\left(\mathbb{Z}\times K_n\right)=4-2\sqrt{3}$.
\end{theorem}
 \noindent In fact, we are able to show the conclusions of Theorems~\ref{theorem: 1-indep percolation in Z^2 times Kn} and~\ref{theorem: answer to Day FR Hancock question} still hold if we replace the complete graph~$K_n$ by a suitable \emph{pseudorandom graph}.  Recall that the study of pseudorandom graphs originates in the ground-breaking work of Thomason~\cite{Thomason87}. In this paper we shall use the following notion of weak pseudorandomness (see Condition (3) in the survey of Krivelevich and Sudakov~\cite{KrivelevichSudakov06}):
\begin{definition}
Let $q=q(n)$ be a sequence in~$[0,1]$. A sequence $(G_n)_{n \in \mathbb{N}}$ of $n$-vertex graphs is \emph{weakly $q$-pseudorandom} if 
\begin{align*}\max \left\{\left\vert e(G_n[U]) - q\frac{\vert U\vert^2}{2}\right\vert: \  U\subseteq V(G_n)\right\}=o(qn^2).\end{align*}	
\end{definition}
\noindent Note that if $(G_n)_{n \in \mathbb{N}}$ is a sequence of weakly $q$-pseudorandom graphs, then for any $U_1, U_2 \subseteq V(G_n)$ with $U_1 \cap U_2 = \varnothing$, we have 
\[e(G_n[U_1, U_2]) = q\s{U_1}\s{U_2} + o(qn^2).\] 
\begin{theorem}\label{theorem: pseudo random graphs} Let $q=q(n)$ satisfy $nq(n)\gg \log n$. Then for any sequence  $(G_n)_{n \in \mathbb{N}}$ of $n$-vertex graphs which is weakly $q$-pseudorandom, we have $\lim_{n\rightarrow \infty}p_{1,c}(\mathbb{Z}^2\times G_n )=4-2\sqrt{3}$.	
\end{theorem}	
\begin{theorem}\label{theorem: pseudo random graphs long paths} Let $q=q(n)$ satisfy $nq(n)\gg \log n$. Then for any sequence  $(G_n)_{n \in \mathbb{N}}$ of $n$-vertex graphs which is weakly $q$-pseudorandom, we have $\lim_{n\rightarrow \infty}p_{1,
	LP}(\mathbb{Z}\times G_n )=4-2\sqrt{3}$.	
\end{theorem}
\noindent We conjecture that the conclusion of Theorem~\ref{theorem: 1-indep percolation in Z^2 times Kn} still holds if we replace the complete graph~$K_n$ by an $n$-dimensional hypercube.
\begin{conjecture}\label{conjecture: Qn}
$\lim_{n\rightarrow \infty}p_{1,c}(\mathbb{Z}^2\times Q_n )=4-2\sqrt{3}$.
\end{conjecture}
\noindent Observe that, since $\mathbb{Z}^2 \times Q_n$ is a subgraph of $\mathbb{Z}^{n+2}$ and $p_{1,c}(\mathbb{Z}^{n+2}) \geq 4-2\sqrt{3}$ \cite[Theorem 7]{DayFalgasRavryHancock20}, Conjecture~\ref{conjecture: Qn} implies that the answer to the problem of Balister and Bollob\'as (Problem~\ref{problem: Balister Bollobas} above) is $4-2\sqrt{3}$. In fact, we make the following bolder conjecture:
\begin{conjecture}[$1$-independent percolation in high dimension]\label{conjecture: Zd} There exists $d\geq 3$ such that
	\[p_{1,c}(\mathbb{Z}^d)=4-2\sqrt{3}.\]	
\end{conjecture}
 \noindent Finally we prove some modest results on component evolution in $1$-independent models on~$K_n$ and on pseudorandom graphs. The main point of these results is that `the two-state measure minimises the size of the largest component', a heuristic which in turn guides our Conjecture~\ref{conjecture: Qn}. Here by the \emph{two-state measure}, we mean the following variant of site percolation, due to Newman (see~\cite{Meester94}):
 \begin{definition}[Two-state measure]
 	Let~$H$ be a graph, and let $p\in [\frac{1}{2}, 1]$. The \emph{two-state measure} $\mu_{2s, p} \in \mathcal{M}_{1,p}(H)$ is constructed as follows: assign to each vertex $v\in V(H)$ a state~$S_v$ independently and uniformly at random, with $S_v=1$ with probability $\theta=\theta(p)=(1+\sqrt{2p-1})/2$ and $S_v=0$ otherwise. Then let $\mathbf{H}_{\mu_{2s, p}}$ be the random subgraph of~$H$ obtained by including an edge if and only if its endpoints are in the same state.
 \end{definition}
\noindent  Day, Hancock and the first author showed in~\cite[Theorem 16]{DayFalgasRavryHancock20} that $\mu_{2s, p}$ minimises the probability of connected subgraphs over all $1$-ipm $\mu \in \mathcal{M}_{1,p}(K_{2n})$. We show below that it also minimises the probability of having a component of size greater than~$n$.  Explicitly, given a set of edges $F\subseteq E(H)$ in a graph~$H$, we let~$C_i(F)$ denote the $i$-th largest connected component in the associated subgraph $(V(H), F)$ of~$H$. Then:	
\begin{proposition}\label{prop: two-state measure minimises prob of compomnent of size 1/2}
	Set $p_{2n}=\frac{1}{2}\left(1-\tan^2\left(\frac{\pi}{4n}\right)\right)$ and $H=K_{2n}$. Then for all $p\in [p_{2n},1]$, 
	\[\min\Bigl\{\mathbb{P}\left[\vert C_1(\mathbf{H}_{\mu})\vert > n \right]: \ \mu \in \M_{1,\geq p}(K_{2n}) \Bigr\}=1-\binom{2n}{n}\left(\frac{1-p}{2}\right)^{n}.\]
\end{proposition}
\noindent Further, we show that the two-state measure also asymptotically minimises the likely size of a largest component in $1$-independent models on pseudorandom graphs:
\begin{theorem}\label{theorem: component size pseudo-random graph}
	Let $r\in \mathbb{N}$, and let $p\in (\frac{1}{r+1}, \frac{1}{r}]$ be fixed. Let $(H_n)_{n\in \mathbb{N}}$ be a sequence of weakly $q$-pseudorandom graphs on~$n$ vertices with $q=q(n)\gg \log(n)/n$. Then the following hold for $H=H_n$:
	\begin{enumerate}[label = \upshape{(\roman*)}]
		\item  For every $\mu \in \mathcal{M}_{1,p}(H)$, with probability~$1-o(1)$ we have $\vert C_1(\mathbf{H}_{\mu})\vert \geq \left(1-o(1)\right)\frac{1+\sqrt{\frac{(r+1)p-1}{r}}}{r+1}n$.
		\item There exists $\mu \in \mathcal{M}_{1,p}(H)$ such that with probability~$1-o(1)$ the random graph $\mathbf{H}_{\mu}$ satisfies $\vert C_1(\mathbf{H}_{\mu})\vert \leq \left(1+o(1)\right)\frac{1+\sqrt{\frac{(r+1)p-1}{r}}}{r+1}n$.
	\end{enumerate}
\end{theorem}
\noindent This leads us to the natural conjecture that the two-state measure asymptotically minimises the size of a largest component in $1$-independent models on the hypercube~$Q_n$:
\begin{conjecture}
	\label{conjecture: Qn component evolution}
	Let $p\in (\frac{1}{2},1]$ be fixed, and let $H=Q_n$. Then for all $\mu \in \mathcal{M}_{1,\geq p}(Q_n)$, with probability $1-o(1)$ we have $ \vert C_1\left(\mathbf{H}_{\mu}\right)\vert \geq \left(\frac{1+\sqrt{2p-1}}{2}-o(1)\right)2^n$.
\end{conjecture}
\noindent We suspect that a proof of this conjecture combined with the ideas in the present paper would yield a proof of Conjecture~\ref{conjecture: Qn}.

Overall, our results would lead us to speculate that the true value of $p_{1,c}(\mathbb{Z}^2)$ is probably a lot closer to the lower bound of~$0.5549$ from~\eqref{eq: lower bound on crit prob from modified site percolation} than to the upper bound of~$0.8639$ obtained from renormalisation arguments in~\cite{BalisterBollobasWalters05}. However a rigorous proof of improved upper bounds on $p_{1,c}(\mathbb{Z}^2)$ remains elusive for the time being.
 
	\subsection{Organisation of the paper}
	The key step in the proof of our main results, Theorem~\ref{theorem: threshold for left meets right}, is proved in Section~\ref{Section: left meets right}; it establishes that $p=4-2\sqrt{3}$ is the threshold for ensuring there is a high probability in any $1$-independent model of finding a path between the largest components in two disjoint copies of~$K_n$ joined by a matching. The argument in a sense captures `what makes the $4-2\sqrt{3}$ measure of~\cite{DayFalgasRavryHancock20, FalgasRavry12} tick'. We then use Theorem~\ref{theorem: threshold for left meets right} in Section~\ref{section: proofs of main theorems} to prove Theorems~\ref{theorem: 1-indep percolation in Z^2 times Kn}--\ref{theorem: pseudo random graphs long paths}. Our component evolution results, Proposition~\ref{prop: two-state measure minimises prob of compomnent of size 1/2} and Theorem~\ref{theorem: component size pseudo-random graph} are proved in Section~\ref{section: component evolution}.
	
	\subsection{Notation}\label{section: notation} 
	Given $n\in \mathbb{N}$ we write~$[n]$ for the discrete interval $\{1,2, \ldots, n\}$. We write~$S^{(2)}$ for the collection of all unordered pairs from a set~$S$. We use standard graph-theoretic notation throughout the paper. Given a graph~$H$, we use $V=V(H)$ and $E=E(H)$ to refer to its vertex-set and edge-set respectively, and write~$e(H)$ for the size of~$E(H)$. Given $X\subseteq V$, we write~$H[X]$ for the subgraph of~$H$ induced by~$X$, i.e.\ the graph $(X, E(H)\cap X^{(2)})$. For disjoint subsets~$X,Y$ of~$V$ we also write~$H[X,Y]$ for the bipartite subgraph of~$H$ induced by $X\sqcup Y$, that is the graph $(X\cup Y, \{xy\in E(H): \ x\in X, y\in Y\})$. We denote by~$K_n$ the complete graph on~$n$ vertices, $K_n=([n], [n]^{(2)})$.

	The \emph{Cartesian product} of two graphs~$G_1$ and~$G_2$ is the graph $G_1\times G_2$ with $V(G_1\times G_2)=\{(v_1, v_2): \ v_1\in V(G_1), v_2\in V(G_2) \}$ and $E(G_1\times G_2)$ consisting of all pairs $\{(u_1,u_2), (v_1,v_2)\}$ with either $u_1=v_1\in V(G_1)$ and $u_2v_2\in E(G_2)$ or $u_1v_1\in E(G_1)$ and $u_2=v_2\in V(G_2)$. In particular if $G_1=K_2$, i.e.\ a single edge, then $G_1\times G_2$ is the \emph{bunkbed graph} of~$G_2$ consisting of two disjoint copies of~$G_2$, the \emph{left copy} $\{1\}\times G_2$ and the \emph{right copy} $\{2\}\times G_2$, together with a perfect matching joining each vertex~$(1,v)$ in the left copy to its image~$(2,v)$ in the right copy. See Figure~\ref{fig:cartesian_product} for an example.
		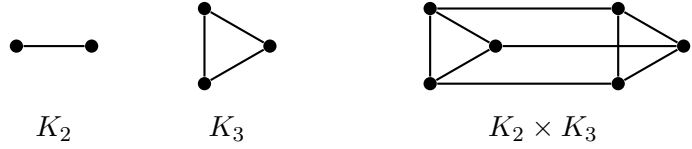
\begin{figure}
			\centering
			\begin{tikzpicture}
			\foreach \pos/ \name in {{(0,0)/a},{(0,1)/b}, {(0.87,0.5)/c}, {(-2.5,0.5)/d}, {(-1.5,0.5)/e},{(0+\vara,0)/aa},{(0+\vara,1)/bb}, {(0.87+\vara,0.5)/cc},{(0+\varb,0)/aaa},{(0+\varb,1)/bbb}, {(0.87+\varb,0.5)/ccc}} 
			\node[vertex] (\name) at \pos {};
			
			\foreach \source/ \dest in {a/b,a/c,b/c,aa/bb,aa/cc,bb/cc,aaa/bbb,aaa/ccc,bbb/ccc,aa/aaa,bb/bbb,cc/ccc,d/e} \path[edge] (\source) -- (\dest);
			
			\node at (-2,-0.6) {$K_2$};
			\node at (0.3,-0.6) {$K_3$};
			\node at (4.5,-0.6) {$K_2 \times K_3$};
			
			\end{tikzpicture}
			\caption{The Cartesian product $K_2 \times K_3$}
			\label{fig:cartesian_product}
		\end{figure}
	
	Finally we use the standard Landau notation for asymptotic behaviour: given functions $f,g: \ \mathbb{N}\rightarrow \mathbb{R}$, we write $f=O(g)$ if $\vert f(n)\vert \leq C\vert g(n)\vert$ for some $C>0$ and all~$n$ sufficiently large, and $f=o(g)$ if $\lim_{n\rightarrow \infty}\vert f(n)/g(n)\vert=0$. We use $f=\Omega(g)$ and $f=\omega(g)$ to denote $g=O(f)$ and $g=o(f)$, respectively. We also sometimes use $f\ll g$ and $f\gg g$ as a shorthand for $f=o(g)$ and $f=\omega(g)$, respectively. Given a sequence of events $\left(E_n\right)_{n\in \mathbb{N}}$ in some probability space, we say that~$E_n$ occurs \emph{with high probability (whp)} if $\mathbb{P}[E_n]=1-o(1)$.

\section{When left meets right: joining the largest components on either side of $K_2\times G_n$}\label{Section: left meets right}	
Let $(G_n)_{n \in \mathbb{N}}$ be a sequence of weakly $q$-pseudorandom $n$-vertex graphs where $qn \gg \log n$.
Consider the Cartesian product $H=K_2\times G_n$. Given $\mu \in \mathcal{M}_{1,p}(H)$, let `$\mathrm{Left\  meets\  Right}$' denote the event that the $\mu$-random graph $\mathbf{H}_{\mu}$ contains a connected component containing both strictly more than half of the vertices in $\{1\}\times [n]$ and strictly more than half of the vertices in $\{2\}\times [n]$. Our main result in this section is showing that the event `$\mathrm{Left\  meets\  Right}$' undergoes a sharp transition at $p=4-2\sqrt{3}$, in the sense that for $p\leq 4-2\sqrt{3}$ it is possible to construct $1$-independent measures $\mu \in \mathcal{M}_{1,p}(H)$ such that whp the event `$\mathrm{Left\  meets\  Right}$' does not occur, while for $p>4-2\sqrt{3}$ it occurs whp regardless of the choice of~$\mu$.	
\begin{theorem}\label{theorem: threshold for left meets right}
\begin{enumerate}[label = \upshape{(\roman*)}]
	\item Let $p>4-2\sqrt{3}$ be fixed. Then for every $\mu \in \mathcal{M}_{1,p}(H)$,
	\[\Pr\left[\mathrm{Left\  meets\  Right}\right]=1-o(1).\] \label{LMR_i}
	\item Let $\frac{1}{2}<p\leq 4-2\sqrt{3}$ be fixed. Then there exists $\mu \in \mathcal{M}_{1,\geq p}(H)$ such that 
		\[\Pr\left[\mathrm{Left\  meets\  Right}\right]=o(1).\] \label{LMR_ii}
\end{enumerate}	
\end{theorem}
\noindent For $p\in (\frac{1}{2},1]$, let $\theta=\theta(p)$ be given by 
\begin{align*}
\theta(p) \coloneqq \frac{1+\sqrt{2p-1}}{2}.
\end{align*}
The quantity~$\theta$ will play an important role in the proof of both parts of Theorem~\ref{theorem: threshold for left meets right}. Observe that $\theta \in [p,1]$ and satisfies 
\begin{align*}
\theta^2+(1-\theta)^2=p \qquad \textrm{ and }\qquad  2\theta(1-\theta)=1-p.\end{align*}
Using the latter of these relations, we see that for $p\in [0,1]$,
\begin{align}
\theta\sqrt{p} \leq 1-p=2\theta(1-\theta)&&&\Leftrightarrow\  p\leq 4(1-\theta)^2= 2p-2\sqrt{2p-1} 
&&\Leftrightarrow\  8p-4\leq p^2 \notag\\ &&&&&\Leftrightarrow\  p\leq 4-2\sqrt{3}.\label{eq: 4-2sqrt 3}\end{align}
\noindent Our proofs will also make extensive use of the following Chernoff bound: given a binomial random variable $X\sim \mathrm{Binom}(N, p)$ and $\varepsilon\in (0,1)$, we have
\begin{align}\label{eq: Chernoff bound}
\mathbb{P}\left[\vert X-Np\vert \geq \varepsilon Np\right]\leq 2 e^{-\frac{\varepsilon^2Np}{3}}.
\end{align}

\subsection{Lower bound construction: proof of Theorem~\ref{theorem: threshold for left meets right}\cref{LMR_ii}}
For each $1/2<p\leq 4-2\sqrt{3}$, we construct a state-based measure $\mu_{F} \in \mathcal{M}_{\geq p}(K_2\times G_n)$, based on the ideas behind constructions in~\cite{DayFalgasRavryHancock20,FalgasRavry12}. Assume without loss of generality that $V(G_n)=[n]$. We randomly assign to each vertex $(i,v)\in [2]\times [n]$ a state~$S_v \in \{0,1,\star\}$, independently of all the other vertices, with 
\begin{enumerate}[label = \upshape{(\alph*)}]
	\item $S_{(1,v)}=1$ with probability~$\theta$ and $S_{(1,v)}=0$ otherwise;
	\item $S_{(2,v)}=0$ with probability~$\sqrt{p}$ and $S_{(2,v)}=\star$ otherwise.
\end{enumerate}
We then include edges of $H=K_2\times G_n$ in our random subgraph $\mathbf{H}_{\mu_F}$ according to the following rules:
\begin{enumerate}[label = \upshape{(\roman*)}]
	\item an edge $\{(1,u), (1,v)\}$ is included if $S_{(1,u)}=S_{(1,v)}$;
	\item an edge $\{(2,u), (2,v)\}$ is included if $S_{(2,u)}=S_{(2,v)}=0$;
	\item an edge $\{(1,v),(2,v)\}$ is included if $S_{(2,v)}=\star$ or if $S_{(1,v)}=S_{(2,v)}=0$.
\end{enumerate}
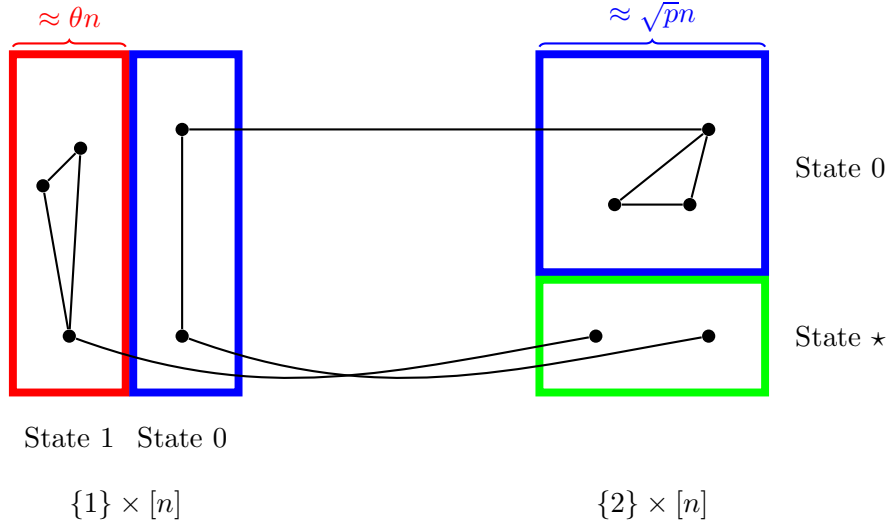
\begin{figure}
	\centering
	\begin{tikzpicture}
	\draw[line width = 3pt, red] (0,0) rectangle (1.5,4.5);
	\draw[line width = 3pt, blue] (1.6,0) rectangle (3,4.5);
	
	\draw[line width = 3pt, blue] (7,1.6) rectangle (10,4.5);
	\draw[line width = 3pt, green] (7,0) rectangle (10,1.5);
	
	\foreach \pos/ \name in {{(0.75,0.75)/a},{(0.4,2.75)/aa},{(0.9, 3.25)/aaa},{(2.25, 0.75)/b},{(2.25, 3.5)/bb},{(9.25, 3.5)/c},{(8, 2.5)/cc},{(9, 2.5)/ccc},{(7.75, 0.75)/d},{(9.25, 0.75)/dd}} 
	\node[vertex] (\name) at \pos {};
	
	\foreach \source/ \dest in {a/aa,aa/aaa,a/aaa,b/bb,c/cc,cc/ccc,c/ccc,bb/c} \path[edge] (\source) -- (\dest);
	
	\path[edge] (a) to[out=-20, in=190] (d);
	\path[edge] (b) to[out=-20, in=190] (dd);
	
	\node at (0.75,-0.6) {State $1$};
	\node at (2.25,-0.6) {State $0$};
	\node at (11,3) {State $0$};
	\node at (11,0.75) {State $\star$};
	
	\node at (1.5,-1.5) {$\{1\} \times [n]$};
	\node at (8.5,-1.5) {$\{2\} \times [n]$};
	
	\draw[decorate, decoration = {brace}, thick,red] (0,4.6) -- (1.5,4.6);
	\node[red] at (0.75,5) {$\approx \theta n$};
	\draw[decorate, decoration = {brace}, thick, blue] (7,4.6) -- (10,4.6);
	\node[blue] at (8.5,5) {$\approx \sqrt{p} n$};
	
	\end{tikzpicture}
	\caption{The lower bound construction}
	\label{fig:lower_bound_construction}
\end{figure}
See \cref{fig:lower_bound_construction} for an illustration of the construction. Since~$\mu_F$ is state-based, it is clearly a $1$-ipm. Our state distributions (a)--(b) imply that every edge in the left copy of~$G_n$ is open (included in our random graph) with probability $\theta^2+(1-\theta)^2=p$ (by the edge-rule (i) above), and that every edge in the right copy of~$G_n$ is open with probability $(\sqrt{p})^2=p$ (by the edge-rule (ii) above). On the other hand, (by the edge-rule (iii) above) an edge $\{(1,v),(2,v)\}$  from the left copy to the right copy is closed if and only if $S_{(1,v)}=1$ and $S_{(2,v)}=0$, which by~\eqref{eq: 4-2sqrt 3} occurs with probability $\theta\sqrt{p}\leq 1-p$ provided $p\leq 4-2\sqrt{3}$. Thus $\mu_{F} \in \mathcal{M}_{1, \geq p}(K_2\times G_n)$ as claimed.

All that remains to show is that  for this measure the event `$\mathrm{Left\  meets\  Right}$' occurs with probability~$o(1)$ in the random graph $\mathbf{H}_{\mu_F}$. 
Observe that the construction of~$\mu_F$ ensures there is no path in $\mathbf{H}_{\mu_F}$ from the vertices in $\{1\}\times [n]$ in state~$1$ to the vertices in $\{2\}\times [n]$ in state~$0$. Indeed the only edges of $\mathbf{H}_{\mu_F}$ in which the endpoints are in different states are those edges containing a vertex~$(2,v)$ in state $S_{(2,v)}=\star$. Since by construction vertices in state~$\star$ have degree exactly one in $\mathbf{H}_{\mu_F}$, it follows that there is no component of~$\mathbf{H}_{\mu_F}$ containing both vertices in state~$1$ and vertices in state~$0$.

Since the expected number of vertices in $\{1\}\times [n]$ in state~$1$ is $\theta n>pn$ and the expected number of vertices in $\{2\}\times [n]$ in state~$0$ is $\sqrt{p}n>pn$, and since states are assigned independently, it follows from~\eqref{eq: Chernoff bound} that for all fixed~$p$ with $1/2<p\leq 4-2\sqrt{3}$, with probability~$1-o(1)$ there is no connected component in~$\mathbf{H}_{\mu_F}$ containing at least half of the vertices of both $\{1\}\times [n]$ and $\{2\}\times [n]$. Thus `$\mathrm{Left\ meets\ Right}$' occurs with probability~$o(1)$ for $\mathbf{H}_{\mu_F}$, as claimed.\qed

\subsection{Upper bound: proof of Theorem~\ref{theorem: threshold for left meets right}\cref{LMR_i}}
Suppose $p>4-2\sqrt{3}$ is fixed. We shall show that for~$n$ sufficiently large this implies that for any $\mu\in \mathcal{M}_{1,p}(H)$, whp `$\mathrm{Left\ meets\ Right}$' occurs.  Our strategy for doing this is as follows: first of all we show in Lemma~\ref{lemma: tripartition probabilistic part} that, for each $i \in [2]$, in any fixed tripartition $\sqcup_{j=1}^3V_j$ of $\{i\}\times [n]$, whp each of the parts~$V_j$ contains roughly the expected number of edges of $\mathbf{H}_{\mu}$, i.e.\ $\left(p+o(1)\right)e(H[V_j])$. This immediately implies that whp there is a component~$C_L$ of $\mathbf{H}_{\mu}$ containing strictly more than half of the vertices of $\{1\}\times [n]$, and another component~$C_R$ containing at least half of the vertices of $\{2\}\times [n]$.

If these two components~$C_L$ and~$C_R$ are not the same, then we colour vertices of $[2]\times [n]$ Green if they lie in a small component of $\mathbf{H}_{\mu}[\{i\}\times [n]]$ for some $i\in [2]$,  and otherwise Red if they are part of~$C_L$ and Blue if not (so in particular vertices in~$C_R$ are coloured Blue). This gives rise to a partition of~$[n]$ into~$9$ sets~$V_{c,c'}$, corresponding to the possible ordered colour pairs assigned to the vertex pairs $((1,v), (2,v))$, $v\in [n]$. Since whp at least~$(p-o(1))n$  of the~$n$ edges from $\{1\}\times [n]$ to $\{2\}\times [n]$ are present in $\mathbf{H}_{\mu}$, we can combine the probabilistic information from Lemma~\ref{lemma: tripartition probabilistic part} to show that whp the relative sizes of the ~$V_{c,c'}$ almost satisfy a certain system $\mathcal{S}=\mathcal{S}(p)$ of inequalities \eqref{eq: positive variables summing to 1}--\eqref{eq: density in the rows} (or more precisely that we can extract from the $\vert V_{c,c'}\vert /n$ a solution to~$\mathcal{S}(p_{\star})$ for some~$p_{\star}$ a little smaller than~$p$). For $p>4-2\sqrt{3}$ and~$n$ sufficiently large, we are able to show this leads to a contradiction (Lemma~\ref{lemma: system has no solution}).
Having outlined our proof strategy, we now fill in the details. We shall use the following path-decomposition theorem due to Dean and Kouider.
\begin{theorem}[Dean and Kouider \cite{DeanKouider00}] \label{thm:path_decomp}
	Let~$G$ be an $n$-vertex graph. Then there exists a set $\mathcal{P}$ of edge-disjoint paths in~$G$ such that $\s{\mathcal{P}} \leq \frac{2n}{3}$ and $\bigcup_{P \in \mathcal{P}} E(P) = E(G)$.
\end{theorem}
\noindent Recall that a \emph{matching} in a graph is a set of vertex-disjoint edges.
\begin{corollary} \label{cor:M_decomp}
	Let $\eps > 0$ and let~$G$ be an $n$-vertex graph with $e(G)\geq 2n/\eps$. Then there exists a set $\mathcal{M}$ of edge-disjoint matchings in~$G$ such that 
	\begin{enumerate}[label = \upshape{(M\arabic*)}, leftmargin= \widthof{M10000}]
		\item $\s{\mathcal{M}} \leq 2n$, \label{M1}
		\item $\s{E(G) \setminus \bigcup_{M \in \mathcal{M}}M} \leq 2\eps e(G)$, and \label{M2}
		\item $\s{M} \geq \frac{\eps e(G)}{2n}$ for every $M \in \mathcal{M}$. \label{M3}
	\end{enumerate}
\end{corollary}
\begin{proof}
	By \cref{thm:path_decomp}, there exists a set $\mathcal{P}$ of edge-disjoint paths in~$G$ such that $\s{\mathcal{P}} \leq \frac{2n}{3}$ and $E(G) = \bigcup_{P \in \mathcal{P}} E(P)$. Let $\mathcal{P}_{\mathrm{short}} = \{P \in \mathcal{P} \colon e(P) \leq 2\eps \frac{e(G)}{n}\}$. Let $\mathcal{M}$ be the set of matchings obtained by decomposing each path in $\mathcal{P} \setminus \mathcal{P}_{\mathrm{short}}$ into two matchings. We have $\s{\mathcal{M}} \leq 2 \s{\mathcal{P}} \leq 2n$. Moreover, each $M \in \mathcal{M}$ satisfies $\s{M} \geq \lfloor\frac{\eps e(G)}{n}\rfloor \geq \frac{\eps e(G)}{2n}$. Finally, $\s{E(G) \setminus \bigcup_{M \in \mathcal{M}} E(M)} \leq \frac{2n}{3} \cdot 2\eps \frac{e(G)}{n} \leq 2\eps e(G)$. 
\end{proof}
\noindent Matchings are useful in a $1$-independent context since the states of their edges (present or absent) are independent. We can thus combine Corollary~\ref{cor:M_decomp} with a Chernoff bound to show the number of edges in a $1$-independent model is concentrated around its mean.
\begin{lemma}\label{lem:edge_concentration}
	Let $\eps > 0$ and $p \in (0,1]$. Let~$G$ be an $n$-vertex graph with $e(G)\geq 2n/\eps$ and let $\mu \in \mathcal{M}_{1,p}(G)$. Then 
	\[
		\mathbb{P}\left[e(\mathbf{G}_{\mu}) \leq (1-3\eps)pe(G)\right] \leq 4n\exp\left(-\frac{\eps^3pe(G)}{6n}\right).
	\]
\end{lemma}
\begin{proof}
	We apply \cref{cor:M_decomp} to obtain a set $\mathcal{M}$ of edge-disjoint matchings in~$G$ such that properties \cref{M1,M2,M3} hold. For every $M \in \mathcal{M}$, we have $\s{M} \geq \frac{\eps e(G)}{2n}$. Thus by \cref{eq: Chernoff bound} and $1$-independence, 
	\begin{align*}
		\prob{e(\Gm \cap M) \leq (1-\eps)p\s{M}} \leq 2\exp\left(-\frac{\eps^3 p e(G)}{6n}\right).
	\end{align*}
	By a union bound, we have 
	\begin{align*}
		\prob{e(\Gm \cap M) \geq (1-\eps)p\s{M} \text{ for  all } M \in \mathcal{M}} &\geq 1-2\vert M\vert \exp\left(-\frac{\eps^3 p e(G)}{6n}\right) \\ 
		&\geq 1- 4n \exp\left(-\frac{\eps^3 p e(G)}{6n}\right).
	\end{align*}
	Thus with probability at least $1- 4n \exp\left(-\frac{\eps^3 p e(G)}{6n}\right)$ we have
	\[
		e(\Gm) \geq \sum_{M \in \mathcal{M}} (1-\eps)p \s{M} \geq (1-\eps)p(1-2\eps)e(G) \geq (1-3\eps)pe(G). 
	\]
	This completes the proof.
\end{proof}

\begin{lemma}\label{lemma: tripartition probabilistic part}
	Let $p \in (\frac{1}{2},1]$, and let $\eps = \eps(p) >0$ be fixed and sufficiently small. Let~$G$ be an $n$-vertex graph satisfying 
\begin{align}\label{assumption: pseudorandomness quant}
\left\vert e(G[U])-q\frac{\vert U\vert^2}{2}\right\vert \leq \frac{\eps^2}{4}qn^2
\end{align}\
	for all $U\subseteq V(G)$, where $q(n) \gg \frac{\log n}{n}$. Consider a fixed tripartition $V(G) = V_1 \sqcup V_2 \sqcup V_3$. Then for every $\mu \in \mathcal{M}_{1,p}(G)$, the following hold whp:
	\begin{enumerate}[label = \upshape{(P\arabic*)}, leftmargin= \widthof{P10000}]
		\item $e(\mathbf{G}_{\mu}[V_i]) \geq pq\frac{\s{V_i}^2}{2} - \eps qn^2$ for every $i \in [3]$. \label{Partition_i}
		\item $e(\mathbf{G}_{\mu}[V_i,V_j]) \geq pq\s{V_i}\s{V_j} - \eps qn^2$ for all $1 \leq i < j \leq 3$.
		\label{Partition_ii}
		\item For every $i \in [3]$ with $\s{V_i} \geq \eps^{1/4}n$, $\mathbf{G}_{\mu}[V_i]$ contains a unique largest connected component~$C_i$ of order at least $(\theta - \eps^{1/4})\s{V_i}$.
		\label{Partition_iii}
		\item For all $1 \leq i < j \leq 3$ with $\s{V_i}, \s{V_j} \geq \eps^{1/4} n$, there exists a path from~$C_i$ to~$C_j$ in $\mathbf{G}_{\mu}[V_i,V_j]$. 
		\label{Partition_iv}
		\item There is a unique largest connected component~$C$ in~$\Gm$ such that $\s{C} \geq (\theta - 3\eps^{1/4})n$ and for each $i \in [3]$ with $\s{V_i} \geq \eps^{1/4}n$, $C_i \subseteq C$.
		\label{Partition_v}
	\end{enumerate}
\end{lemma}
\begin{proof}
	We first show that \cref{Partition_i} holds whp. Fix $i \in [3]$. If $\s{V_i} \leq \sqrt{\eps} n$, then \cref{Partition_i} trivially holds. Hence we assume that $\s{V_i} \geq \sqrt{\eps} n$. By our pseudorandomness assumption~\eqref{assumption: pseudorandomness quant} on~$G$ we have $e(G[V_i]) \geq q\frac{\s{V_i}^2}{2} -\frac{\eps}{2} qn^2$ (which for~$n$ sufficiently large is greater than $\frac{2n}{\eps}$ so that we can apply \cref{lem:edge_concentration}). Thus we have
	\begin{align*}
		\prob{e(\mathbf{G}_{\mu}[V_i]) \leq pq\frac{\s{V_i}^2}{2} - \eps qn^2}
		&\leq \prob{ e(\mathbf{G}_{\mu}[V_i]) \leq pe(G[V_i]) - \frac{\eps}{2} qn^2 } \\
		&\leq \prob{e(\mathbf{G}_{\mu}[V_i]) \leq \left(1- \frac{\eps}{3}\right)pe(G[V_i])} \\
		&\leq 4n\exp\left(-\Omega\left(\frac{e(G[V_i])}{n}\right)\right) = 4n\exp\left(-\Omega(qn)\right)=o(1),
	\end{align*}
where the inequality in the third  line follows from \cref{lem:edge_concentration}. So \cref{Partition_i} holds whp. 

Next we show that \cref{Partition_ii} holds whp. Fix $1 \leq i < j \leq 3$. If $\s{V_i} \leq \eps n$ or $\s{V_j} \leq \eps n$, then \cref{Partition_ii} trivially holds. Hence we may assume that $\s{V_i}, \s{V_j} \geq \eps n$. By~\eqref{assumption: pseudorandomness quant} applied three times (to~$V_i$,~$V_j$ and $V_i\cup V_j$), we have $e(G[V_i, V_j]) \geq q\s{V_i}\s{V_j} - 3\frac{\varepsilon^2}{4} qn^2$. In particular, $e(G[V_i,V_j]) \geq \frac{\varepsilon^2}{4} qn^2$, which for~$n$ sufficiently large is greater than $\frac{2n}{\varepsilon}$. We now apply \cref{lem:edge_concentration} to show that \cref{Partition_ii} holds whp. We have
\begin{align*}
	\prob{e(\mathbf{G}_{\mu}[V_i, V_j]) \leq pq\s{V_i}\s{V_j} - \eps qn^2} 
	&\leq \prob{e(\mathbf{G}_{\mu}[V_i, V_j]) \leq pe(G[V_i,V_j]) - \frac{\eps}{2} qn^2} \\
	&\leq \prob{e(\mathbf{G}_{\mu}[V_i,V_j]) \leq \left(1- \frac{\eps}{3}\right)pe(G[V_i,V_j])} \\
	&\leq 4n\exp\left(-\Omega\left(\frac{e(G[V_i,V_j])}{n}\right)\right) =4n\exp\left(-\Omega(qn)\right)=o(1).
\end{align*}
So \cref{Partition_ii} holds whp. 

Now we show that \cref{Partition_i} implies \cref{Partition_iii}. Assume that \cref{Partition_i} holds. Fix $i \in [3]$ and assume that $\s{V_i} \geq \eps^{1/4}n$. Let $C \subseteq V_i$ be a largest connected component in~$\Gm[V_i]$ and suppose for a contradiction that $\s{C} \leq (\theta - \eps^{1/4})\s{V_i}$.

If $\s{C}\leq \frac{\s{V_i}}{2}$, then there is a partition of~$V_i$ into at most~$4$ sets, each of size at most $\frac{\s{V_i}}{2}$, such that every connected component of~$\Gm[V_i]$ is entirely contained in one of the sets of the partition. Indeed, such a partition can be obtained by starting with a partition of $V_i$ into the connected components of~$\Gm[V_i]$ and then as long as the partition contains two parts of size at most $\frac{\s{V_i}}{4}$ choosing two such parts arbitrarily  and merging them into a single part. Since for any quadruple $(x_1,x_2, x_3, x_4)$ with $\frac{1}{2}\geq x_i\geq 0$ and $\sum_i x_i=1$ we have $\sum_i (x_i)^2\leq \frac{1}{2}$, it follows from~\cref{Partition_i} and~\eqref{assumption: pseudorandomness quant} that
\begin{align*}
pq\frac{\s{V_i}^2}{2}-\eps qn^2 &\leq e(\Gm[V_i])\leq q \frac{\s{V_i}^2}{4}+\varepsilon^2qn^2.
\end{align*}
Rearranging terms, this gives
\begin{align*}
(p-\frac{1}{2})q\frac{\eps^{1/2}n^2}{2}\leq (p-\frac{1}{2})q\frac{\s{V_i}^2}{2}\leq q(\varepsilon+\varepsilon^2)n^2,
\end{align*}
which is a contradiction for~$\eps$ chosen sufficiently small. Thus we may assume $\s{C}\geq \frac{\s{V_i}}{2}$. Now by \cref{Partition_i} and~\eqref{assumption: pseudorandomness quant} again, we have
\begin{align*}
	pq\frac{\s{V_i}^2}{2}-\eps qn^2 &\leq e(\Gm[V_i]) \leq e(\Gm[C]) + e(\Gm[V_i\setminus C]) \leq q\frac{\s{C}^2}{2} + q \frac{(\s{V_i}-\s{C})^2}{2} +\frac{\varepsilon^2}{2}qn^2.
\end{align*} 
Dividing by $q\frac{\s{V_i}^2}{2}$ and using $\s{V_i} \geq \eps^{1/4}n$, we deduce that 
\begin{align}
	\label{eq:C_bound}
	p - 3\sqrt{\eps} \leq \left(\frac{\s{C}}{\s{V_i}}\right)^2 + \left(1-\frac{\s{C}}{\s{V_i}}\right)^2.
\end{align}
Since $x \mapsto x^2 + (1-x)^2$ is an increasing function in the interval $[\frac{1}{2}, 1]$, $\frac{1}{2}\s{V_i} \leq \s{C} \leq (\theta - \eps^{1/4})\s{V_i}$, and $\theta^2 + (1-\theta)^2 =p$, we have
\begin{align*}
	\left(\frac{\s{C}}{\s{V_i}}\right)^2 + \left(1-\frac{\s{C}}{\s{V_i}}\right)^2 &\leq (\theta - \eps^{1/4})^2 + (1- \theta + \eps^{1/4})^2 \\
	&= \theta^2 +(1-\theta)^2 - 2\eps^{1/4}(2\theta -1) + 2\sqrt{\eps} \leq p - 4 \sqrt{\eps},
\end{align*}
contradicting  \cref{eq:C_bound}. Hence $\s{C} \geq (\theta - \eps^{1/4})\s{V_i}$. Note that since $\theta - \eps^{1/4} > 1/2$ (for $\eps=\eps(p)$ chosen sufficiently small),~$C$ is the unique largest component in~$\Gm[V_i]$. 
So \cref{Partition_iii} holds whp.

Next we show that \cref{Partition_ii} and \cref{Partition_iii} together imply \cref{Partition_iv}. Assume that \cref{Partition_ii} and \cref{Partition_iii} hold. Fix $1 \leq i < j \leq 3$ and assume that $\s{V_i}, \s{V_j} \geq \eps^{1/4}n$. Suppose for a contradiction that there is no path in $\Gm[V_i,V_j]$ from~$C_i$ to~$C_j$. Let $A_i \subseteq V_i$ and $A_j \subseteq V_j$ be the sets of vertices which cannot be reached by a path in $\Gm[V_i,V_j]$ from~$C_j$ and~$C_i$, respectively. Since there is no path from~$C_i$ to~$C_j$, we must have $C_i \subseteq A_i$ and $C_j \subseteq A_j$. By~\cref{Partition_ii}, by the definition of~$A_i$ and~$A_j$, and by~\eqref{assumption: pseudorandomness quant} (applied in~$A_i$, $A_j$, $V_i\setminus A_i$, $V_j\setminus A_j$, $A_i\cup(V_j\setminus A_j )$ and $A_j\cup(V_i\setminus A_i)$), we have
\begin{align} \label{eq:A_bound}
\begin{aligned}
	pq\s{V_i}\s{V_j} - \eps qn^2 \leq e(\Gm[V_i, V_j]) &\leq e(\Gm[A_i,V_j\setminus A_j]) + e(\Gm[V_i \setminus A_i, A_j]) \\
	&\leq q \s{A_i} (\s{V_j}-\s{A_j}) + q \s{A_j} (\s{V_i}-\s{A_i}) +\frac{3\eps^2}{2}qn^2.
\end{aligned}
\end{align}
Let $x_i = \frac{\s{A_i}}{\s{V_i}}$ and $x_j = \frac{\s{A_j}}{\s{V_j}}$. By \cref{Partition_iii}, $x_i \geq \frac{\s{C_i}}{\s{V_i}} \geq \theta - \eps^{1/4} \geq \frac{1}{2}$ and similarly $x_j \geq \frac{1}{2}$. From \cref{eq:A_bound} we get by dividing by $q\s{V_i}\s{V_j}$ and using $\s{V_i}, \s{V_j} \geq \eps^{1/4}n$, that 
\begin{align}\label{eq:Contr_iv}
	p - 2 \sqrt{\eps} \leq x_i(1-x_j) + x_j(1-x_i) = x_i + x_j -2x_ix_j \leq \frac{1}{2},
\end{align}
where the last inequality follows since $(x,y) \mapsto x+y-2xy$ is non-increasing in both~$x$ and~$y$ for $x,y \geq \frac{1}{2}$. Note that \cref{eq:Contr_iv} gives a contradiction for~$\eps$ sufficiently small since $p > \frac{1}{2}$. So \cref{Partition_iv} holds whp.

\noindent Finally, we observe that \cref{Partition_v} follows directly from \cref{Partition_iii} and \cref{Partition_iv}. Indeed let $k\in [3]$ denote the number of $i\in [3]$ for which $\vert V_i\vert \geq \eps^{1/4}n$ (note we can guarantee $k\geq 1$ provided $\eps<3^{-4}$). Then \cref{Partition_iii} and \cref{Partition_iv} together imply there is a unique connected component $C$ in $\mathbf{G}_{\mu}$ of size at least $(\theta-\eps^{1/4})(1-(3-k)\eps^{1/4})n>(\theta -3\eps^{1/4})n$ and containing $C_i$ for each $i\in [3]$ with $\vert V_i\vert \geq \eps^{1/4}n$. 
\end{proof}
\noindent Let $\mathcal{S}(p)$ denote the collection of $3\times 3$ matrices~$A$ with non-negative entries $A_{ij}\geq 0$, $i,j\in [3]$, satisfying the following inequalities:
\begin{align}
 A_{11}+A_{22}+p\leq \sum_{i,j} A_{ij} \leq 1 \label{eq: positive variables summing to 1} \\ 
A_{1j} \geq \frac{1}{2}\sum_i A_{ij} \quad  \forall j\in [3] \qquad \textrm{ and }\qquad   A_{i1} \geq \frac{1}{2}\sum_j A_{ij} \quad  \forall i \in [3]\label{eq: columns and rows respect giants}\\
\left(A_{1j}\right)^2+\left(A_{2j}\right)^2\geq p\left( \sum_i A_{ij}\right)^2 \quad \forall j\in [3] \label{eq: density in the columns} \\
\left(A_{i1}\right)^2+\left(A_{i2}\right)^2\geq p\left( \sum_j A_{ij}\right)^2 \quad \forall i\in [3]\label{eq: density in the rows}
\end{align}
The key step in our proof of Theorem~\ref{theorem: threshold for left meets right} will be, assuming that `$\mathrm{Left\  meets\  Right}$' does not occur whp, to use Lemma~\ref{lemma: tripartition probabilistic part} to exhibit a partition of~$[n]$ into~$9$ parts whose relative sizes can be used to find a solution to~$\mathcal{S}(p_{\star})$, for some~$p_{\star}$ satisfying $4-2\sqrt{3}<p_{\star}<p$. We will then be able to use the following lemma to derive a contradiction. 
\begin{lemma}\label{lemma: system has no solution}
For $4-2\sqrt{3}<p\leq 1$, $\mathcal{S}(p) = \varnothing$.	
\end{lemma}
\begin{proof}
Suppose not and let $A \in \mathcal{S}(p)$. Note that the bound for $\sum_{i,j}A_{ij}$ in~\eqref{eq: positive variables summing to 1} implies
\begin{align}
 A_{11}+A_{22}\leq 1-p. \label{eq: no crossing}
\end{align}
By transpose-symmetry of $\mathcal{S}(p)$ and \cref{eq: positive variables summing to 1}, we may assume without loss of generality that 
\begin{equation}
	\label{eq:w_bound}
	w \coloneqq A_{21}+A_{31}+ A_{32}+A_{33}\geq \frac{p}{2}.
\end{equation}
Note that if $\sum_{j} A_{3j} > \frac{A_{31}}{\theta}$, then, since $x \mapsto x^2 +(1-x)^2$ is an increasing function of~$x$ in the interval $[\frac{1}{2}, 1]$ and since $A_{31}\geq \frac{1}{2}\sum_j A_{33j}$ by~\eqref{eq: columns and rows respect giants},
\[
	\left(\frac{A_{31}}{\sum_{j} A_{3j}}\right)^2 + \left(\frac{A_{32}}{\sum_{j} A_{3j}}\right)^2 \leq \left(\frac{A_{31}}{\sum_{j} A_{3j}}\right)^2 + \left(1-\frac{A_{31}}{\sum_{j} A_{3j}}\right)^2 < \theta^2 + (1-\theta)^2 = p,
\]
contradicting \cref{eq: density in the rows}. 
Hence 
\begin{equation}
	\label{eq:Row_bound}
	\sum_{j} A_{3j} \leq \frac{A_{31}}{\theta}.
\end{equation}
By an analogous argument, we have $\sum_{i} A_{i1} \leq \frac{A_{11}}{\theta}$ and thus
\begin{equation}
	\label{eq:A_22_bound}
	A_{21} \leq A_{21} + A_{31} \leq \frac{1-\theta}{\theta}A_{11}.
\end{equation}
Now, by \cref{eq:Row_bound} we have $w \leq A_{21} + \frac{A_{31}}{\theta}$. By \cref{eq: density in the columns}, we have that 
\[
	A_{31} \leq \frac{\sqrt{(A_{11})^2+(A_{21})^2}}{\sqrt{p}} -A_{11}-A_{21}.
\]
Substituting this expression into our upper bound on~$w$, we get 
\[
	w \leq -\frac{(1-\theta)A_{21} }{\theta}  -\frac{A_{11}}{\theta} +\frac{\sqrt {(A_{11})^2+(A_{21})^2 }}{\theta\sqrt{p}}.
\]
For~$A_{11}$ fixed, the continuous function $f_{A_{11}}(y) = -\frac{(1-\theta)y }{\theta}  -\frac{A_{11}}{\theta} +\frac{\sqrt {(A_{11})^2+y^2 }}{\theta\sqrt{p}}$ is convex in $(0, +\infty)$ as its derivative 
$f_{A_{11}}'(y) = - \frac{(1-\theta)}{\theta} +\frac{1}{\theta\sqrt{p}\sqrt{(A_{11}/y)^2+1}}$
is increasing in~$y$ in that interval. By \cref{eq:A_22_bound}, $0 \leq A_{21} \leq \frac{1-\theta}{\theta}A_{11}$, which together with the convexity of~$f_{A_{11}}$ gives:
\begin{align*}
	w &\leq \max\left\{f_{A_{11}}(0),\  f_{A_{11}}\left(\frac{1-\theta}{\theta}A_{11}\right)\right\} \\
	&\leq \max\left\{-\frac{A_{11}}{\theta} + \frac{A_{11}}{\theta \sqrt{p}},\  -\left(\frac{1- \theta}{\theta}\right)^2A_{11} -\frac{A_{11}}{\theta} + A_{11}\frac{\sqrt{1 + \left(\frac{1- \theta}{\theta}\right)^2}}{\theta \sqrt{p}}\right\} \\
	&\leq \max \left\{ \frac{A_{11}}{\theta}\left(\frac{1}{\sqrt{p}}-1\right),\  \frac{A_{11}}{\theta}(1-\theta)\right\} \\
	&\leq \max \left\{ \frac{1-p}{\theta}\left(\frac{1}{\sqrt{p}}-1\right),\  \frac{1-p}{\theta}(1-\theta)\right\},
\end{align*}
where the last inequality follows from the upper bound \cref{eq: no crossing} on~$A_{11}$. We now claim that this contradicts~\cref{eq:w_bound}, i.e.\ that
\[
	\max \left\{ \frac{1-p}{\theta}\left(\frac{1}{\sqrt{p}}-1\right), \frac{1-p}{\theta}(1-\theta)\right\} < \frac{p}{2}.
\]
Note that $p \mapsto \frac{1-p}{\theta}\left(\frac{1}{\sqrt{p}}-1\right) - \frac{p}{2}$ and $p \mapsto \frac{1-p}{\theta}(1-\theta) - \frac{p}{2}$ are both strictly decreasing functions (as~$\theta$ is increasing in~$p$). Hence to prove the claim above, it suffices to show that for $p = 4-2\sqrt{3}$, we have $\frac{1-p}{\theta}\left(\frac{1}{\sqrt{p}}-1\right) \leq \frac{p}{2}$ and $\frac{1-p}{\theta}(1-\theta) \leq \frac{p}{2}$. Let $p = 4-2\sqrt{3}$. Note that $(\sqrt{3}-1)^2 = 4-2\sqrt{3}$ and $(2-\sqrt{3})^2 = 7-4\sqrt{3}$. Hence $\sqrt{p} = \sqrt{3}-1$, $\sqrt{2p-1} = 2-\sqrt{3}$, and $\theta = (3-\sqrt{3})/2$. Now it is easy to check that 
\[
	\frac{1}{\sqrt{p}} -1 = 1 - \theta = \frac{\theta}{(1-p)}\frac{p}{2}=\frac{\sqrt{3}-1}{2},
\]
which completes the proof.
\end{proof}
\noindent We are now ready to complete the proof of \cref{theorem: threshold for left meets right} \cref{LMR_i}.
\begin{proof}
	Let $p > 4-2\sqrt{3}$ be fixed. Let $\eps = \eps(p) > 0$ be fixed and chosen sufficiently small. Let $p_{\star} =\frac{1}{2}\left(4-2\sqrt{3}+p\right)$. 	
	Finally, let~$n$ be sufficiently large so that for $G=G_n$ the pseudorandomness assumption~\eqref{assumption: pseudorandomness quant} holds, and let $\mu \in \mathcal{M}_{1,p}(H)$, where $H=K_2\times G_n$.

	For $i \in [2]$, let $\Gm^i = \Hm[\{i\} \times [n]]$. For $i,j \in [2]$ with $i \neq j$, let $\mathcal{E}_{ij}$ be the event that for any partition $(\{i\}\times V_1) \sqcup (\{i\}\times V_2) \sqcup (\{i\}\times V_3)$ of $\{i\}\times [n]$ such that $\{i\} \times V_1$ and $\{i\} \times V_2$ are each a union of components of order at least $\eps^{1/4}n$ in~$\Gm^i$, we have that~$\Gm^j$ satisfies \cref{Partition_i,Partition_ii,Partition_iii,Partition_iv,Partition_v} of \cref{lemma: tripartition probabilistic part} with $\{j\} \times V_1$, $\{j\} \times V_2$, $\{j\} \times V_3$ playing the roles of~$V_1$, $V_2$, $V_3$. Given~$\Gm^i$ and~$\eps$ fixed, the number of such partitions is at most $3^{\eps^{-1/4}}=O(1)$. Hence \cref{lemma: tripartition probabilistic part} implies that $\mathcal{E}_{ij}$ holds whp.

Further, by $1$-independence and~\eqref{eq: Chernoff bound}, whp  there are at least $(p - \eps)n$ edges in the matching $\Hm[\{1\} \times [n], \{2\} \times [n]]$. Let $\mathcal{E}_{\mathrm{good}}$ be the event that $\mathcal{E}_{12}$ and $\mathcal{E}_{21}$ both occur and that in addition $e(\Hm[\{1\} \times [n], \{2\} \times [n]])\geq  (p - \eps)n$. Then $\mathcal{E}_{\mathrm{good}}$ holds whp. We claim that if  $\mathcal{E}_{\mathrm{good}}$ holds, then so does `$\mathrm{Left\  meets\  Right}$' (which implies the statement of the theorem).

	Suppose for a contradiction that $\mathcal{E}_{\mathrm{good}}$ holds but `$\mathrm{Left\  meets\  Right}$' does not. For $i \in [2]$, let~$C^i$ be the unique largest connected component in~$\Gm^i$ (this exist by \cref{Partition_v}). 
	Let $U_1 \sqcup U_2 \sqcup U_3 = [n]$ and $W_1 \sqcup W_2 \sqcup W_3 = [n]$ be such that the following hold.
	\begin{enumerate}[label = (\upshape{\alph*})]
		\item $\{1\} \times U_1$ is the union of~$C^1$ and all connected components in~$\Gm^1$ of order at least $\eps^{1/4}n$ that can be reached from~$C^1$ by a path in~$\Hm$.
		\item $\{1\} \times U_2$ is the union of all other connected components in~$\Gm^1$ of order at least $\eps^{1/4}n$. 
		\item $\{1\} \times U_3$ is the union of all connected components of order less than $\eps^{1/4}n$ in~$\Gm^1$.
		\item $\{2\} \times W_1$ is the union of all connected components in~$\Gm^2$ of order at least $\eps^{1/4}n$ that cannot be reached from~$C^1$ by a path in~$\Hm$.
		\item $\{2\} \times W_2$ is the union of all connected components in~$\Gm^2$ of order at least $\eps^{1/4}n$ that can be reached from~$C^1$ by a path in~$\Hm$.
		\item $\{2\} \times W_3$ is the union of all connected components in~$\Gm^2$ of order less than $\eps^{1/4}n$.
	\end{enumerate}
We can think of these partitions as giving us a $3$-colouring of the vertices in~$V(H)$: a vertex in $\{i\}\times V_n$ is coloured red if it belongs to a large component in $\mathbf{G}^i_{\mu}$ and can be reached from~$C^1$ in $\mathbf{H}_{\mu}$, blue if it belongs to a large component in $\mathbf{G}^i_{\mu}$ and cannot be reached by~$C^1$ in $\mathbf{H}_{\mu}$, and green if it belongs to a small component in $\mathbf{G}^i_{\mu}$. The key properties of this colouring are that the large components~$C^1$ and~$C^2$ in $\mathbf{G}^1_{\mu}$ and   $\mathbf{G}^2_{\mu}$ are coloured red and blue respectively, that there are no edges from red vertices to blue vertices, and that the green vertices span few edges in $\mathbf{G}^i_{\mu}$, $i\in [2]$. Our $3$-colouring of~$V(H)$ gives rise to a partition of~$[n]$ into~$9$ sets in a natural way, by considering the possible colour pairs for $((1,v), (2,v))$, $v\in [n]$. This partition is illustrated in \cref{fig:partition}.

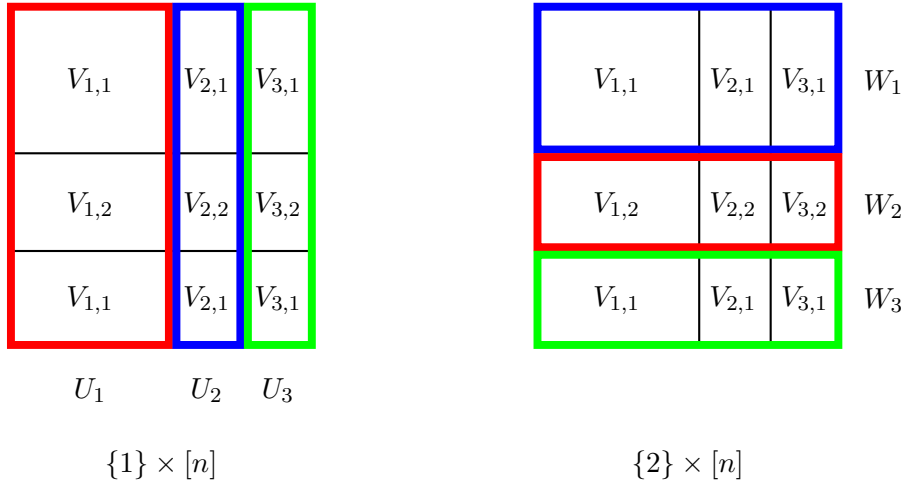
\begin{figure}
	\centering
	\begin{tikzpicture}
	\path[edge] (-1,1.25) -- (3,1.25);
	\path[edge] (-1,2.55) -- (3,2.55);
	\path[edge] (8.15,0) -- (8.15,4.5);
	\path[edge] (9.1,0) -- (9.1,4.5);
	
	\draw[line width = 3pt, red] (-1,0) rectangle (1.1,4.5);
	\draw[line width = 3pt, blue] (1.2,0) rectangle (2.05,4.5);
	\draw[line width = 3pt, green] (2.15,0) rectangle (3,4.5);
	\draw[line width = 3pt, blue] (6,2.6) rectangle (10,4.5);
	\draw[line width = 3pt, red] (6,1.3) rectangle (10,2.5);
	\draw[line width = 3pt, green] (6,0) rectangle (10,1.2);

	\node at (0.05,3.5) {$V_{1,1}$};
	\node at (1.6,3.5) {$V_{2,1}$};
	\node at (2.55,3.5) {$V_{3,1}$};
	\node at (0.05,1.85) {$V_{1,2}$};
	\node at (1.6,1.85) {$V_{2,2}$};
	\node at (2.55,1.85) {$V_{3,2}$};
	\node at (0.05,0.6) {$V_{1,1}$};
	\node at (1.6,0.6) {$V_{2,1}$};
	\node at (2.55,0.6) {$V_{3,1}$};
	
	\node at (0.05+7,3.5) {$V_{1,1}$};
	\node at (1.6+7,3.5) {$V_{2,1}$};
	\node at (2.55+7,3.5) {$V_{3,1}$};
	\node at (0.05+7,1.85) {$V_{1,2}$};
	\node at (1.6+7,1.85) {$V_{2,2}$};
	\node at (2.55+7,1.85) {$V_{3,2}$};
	\node at (0.05+7,0.6) {$V_{1,1}$};
	\node at (1.6+7,0.6) {$V_{2,1}$};
	\node at (2.55+7,0.6) {$V_{3,1}$};
	
	\node at (0.05,-0.6) {$U_1$};
	\node at (1.6,-0.6) {$U_2$};
	\node at (2.55,-0.6) {$U_3$};
	
	\node at (10.6,3.5) {$W_1$};
	\node at (10.6,1.85) {$W_2$};
	\node at (10.6,0.6) {$W_3$};
	
	\node at (1,-1.6) {$\{1\} \times [n]$};
	\node at (8,-1.6) {$\{2\} \times [n]$};
	
	\end{tikzpicture}
	\caption{The partition of $V(H)$}
	\label{fig:partition}
\end{figure}

We now investigate the relative sizes of this $9$-partition. For $i,j \in [3]$, let $V_{ij} = U_i \cap W_j$. Since there is no path from~$C^1$ to~$C^2$ in~$\Hm$, there are no edges present in the bipartite graphs $\Hm[\{1\} \times V_{11}, \{2\} \times V_{11}]$ and $\Hm[\{1\} \times V_{22}, \{2\} \times V_{22}]$. Since $\mathcal{E}_{\mathrm{good}}$ holds, there are at least~$(p-\eps)n$ edges in $\Hm[\{1\} \times [n], \{2\} \times [n]]$ in total, which implies
	\begin{align} \label{eq:A_1}
		\s{V_{11}} + \s{V_{22}} \leq (1-p+\eps)n. 
	\end{align}
	Moreover, $\sum_{i,j} \s{V_{ij}} = n$. Hence
	\begin{align} \label{eq:A_2}
		\sum_{i,j} \s{V_{ij}} - \s{V_{11}} - \s{V_{22}} \geq (p-\eps)n.
	\end{align}
	For $j\in [3]$, if $\vert W_j\vert \geq  \eps^{1/4}n$, we have by \cref{Partition_iii,Partition_v} that there is a unique largest connected component~$C_j^1$ in $\Gm^1[\{1\} \times W_j]$, and that this component satisfies  $C_j^1\subseteq C^1$ and $\vert C_j^1\vert \geq (\theta -\eps^{1/4})\vert W_j\vert$, which for $\eps=\eps(p)$ chosen sufficiently small is greater than $\frac{1}{2}\vert W_j\vert$. Translating this in terms of our $9$-partition, we have that for all $j\in [3]$ such that $\sum_i V_{ij}\geq \eps^{1/4}n$
	\begin{align} \label{eq:A_3}
		\s{V_{1j}} \geq \frac{1}{2} \sum_i\s{V_{ij}}
	\end{align}
	holds. By a symmetric argument, for every $i \in [3]$ such that $\sum_j V_{ij}\geq \eps^{1/4}n$ we have
	\begin{align} \label{eq:A_4}
		\s{V_{i1}} \geq  \frac{1}{2}\sum_j \s{V_{ij}}.
	\end{align}
	Let $j \in [3]$. Note that~$\Gm^1[U_3]$ contains only connected components of size at most $\eps^{1/4}n$. These components can be covered by at most $\frac{2}{\eps^{1/4}}$ sets, each of order at least $\frac{\eps^{1/4}n}{2}$ and at most $\eps^{1/4}n$. By~\eqref{assumption: pseudorandomness quant} (which holds by our choice of~$n$), each of these sets contains at most $q\frac{\eps^{1/2}n^2}{2} +\frac{\varepsilon^2}{4}qn^2<q \eps^{1/2}n^2$ edges. Hence we have $e(\Gm^1[U_3]) \leq 2\eps^{1/4}qn^2$. Since $V_{3j} \subseteq U_3$, we have $e(\Gm^1[V_{3j}]) \leq 2\eps^{1/4}qn^2$. By \cref{Partition_i} and the pseudorandomness assumption~\eqref{assumption: pseudorandomness quant}, we have 
	\begin{align*}
		pq\frac{\s{W_j}^2}{2} - \eps qn^2 &\leq e(\Gm^1[\{1\} \times W_j])\\
		 &= e(\Gm^1[\{1\} \times V_{1j}]) + e(\Gm^1[\{1\} \times V_{2j}]) + e(\Gm^1[\{1\} \times V_{3j}]) \\
		&\leq q\frac{\s{V_{1j}}^2}{2} + q\frac{\s{V_{2j}}^2}{2} + 2\eps^{1/4}qn^2 + \frac{\varepsilon^2}{2}qn^2<q\frac{\s{V_{1j}}^2}{2} + q\frac{\s{V_{2j}}^2}{2} + 3\eps^{1/4}qn^2.
	\end{align*}
	Hence, for every $j \in [3]$ and~$\eps$ chosen sufficiently small,
	\begin{align} \label{eq:A_5}
		\s{V_{1j}}^2 + \s{V_{2j}}^2 \geq p\left(\sum_i \s{V_{ij}}\right)^2 - 7\eps^{1/4}n^2.
	\end{align}
	Similarly, for every $i \in [3]$,
	\begin{align} \label{eq:A_6}
		\s{V_{i1}}^2 + \s{V_{i2}}^2 \geq p\left(\sum_j \s{V_{ij}}\right)^2 - 7\eps^{1/4}n^2.
	\end{align}
	Let~$A$ be the $3 \times 3$ matrix with entries
	\begin{align*}
		A_{ij} = 
		\begin{cases}
			\frac{\s{V_{ij}}}{n}, &\text{ if } \s{V_{ij}} \geq \eps^{1/9}n, \\
			0, &\text{ otherwise}.
		\end{cases}
	\end{align*}
	We claim that, provided $\eps=\eps(p)$ was chosen sufficiently small, $A \in \mathcal{S}(p_{\star})$. Indeed,~$A$ clearly has nonnegative entries summing up to at most~$1$, thus the second inequality of~\eqref{eq: positive variables summing to 1} is satisfied,  while the first inequality (with~$p_{\star}$ instead of~$p$) follows from~\eqref{eq:A_2} and an appropriately small choice of~$\eps$ (more specifically, we need $p_{\star}\leq p-\varepsilon  - 7\varepsilon^{1/9}$). Indeed,
		\begin{align*}
		\sum_{i,j} A_{ij} - A_{11} - A_{22} \geq \sum_{i,j} \frac{\s{V_{ij}}}{n} - \frac{\s{V_{11}}}{n} -\frac{\s{V_{22}}}{n} - 7\eps^{1/9} \geq p - \eps - 7 \eps^{1/9} \geq p_{\star},
		\end{align*}	
		where the penultimate inequality uses~\eqref{eq:A_2}.

	Next, consider $j\in [3]$. If $\sum_i \s{V_i}\geq \eps^{1/4}n$, then by~\eqref{eq:A_3} we have $A_{1j}\geq \frac{1}{2}\sum_i A_{ij}$ (regardless of whether some of the~$V_{ij}$, $i\in [3]$ have size less than $\varepsilon^{1/9}n$). Other the other hand if $\sum_i \s{V_i}< \eps^{1/4}n$, then $A_{1j}=A_{2j}=A_{3j}=0$. In either case, $A_{1j}\geq \frac{1}{2}\sum_i A_{ij}$ holds. By a symmetric argument we obtain that $A_{i1}\geq \frac{1}{2}\sum_j A_{ij}$ holds for every $i\in [3]$. Thus~\eqref{eq: columns and rows respect giants} is satisfied  by~$A$.

	Finally, pick $j\in [3]$. If $\s{V_{i2}}\geq \eps^{1/9}n$, then by~\eqref{eq: columns and rows respect giants} which we have just established and the definition of~$A_{i1}$, we have $\s{V_{i1}}\geq \eps^{1/9}n$ also. In this case~\eqref{eq:A_5} and an appropriately small choice of~$\eps$ ensure that $(A_{1j})^2+ (A_{2j})^2 \geq p_{\star} \left(\sum_i A_{ij}\right)^2$. On the other hand, suppose $\s{V_{i2}}< \eps^{1/9}n$. If $\s{V_{i1}}< \eps^{1/9}n$, then by~\eqref{eq: columns and rows respect giants} the inequality $(A_{1j})^2+ (A_{2j})^2 \geq p_{\star} \left(\sum_i A_{ij}\right)^2$ holds trivially, since the right hand-side is zero. So suppose that $\s{V_{i1}}\geq \eps^{1/9}n>\s{V_{i2}}$. Then~\eqref{eq:A_5}, and $p>1/2$ imply that
	\begin{align*}
	\s{V_{i1}}^2>\s{V_{i1}}^2 - \s{V_{i2}}\left(2p\s{V_{i1}}  -(1-p)\s{V_{i2}}\right)\geq p\left(\s{V_{i1}}+\s{V_{i3}}\right)^2-7\varepsilon^{1/4}n^2.
\end{align*}
Together with an appropriately small choice of~$\eps$, this ensures $(A_{1j})^2+ (A_{2j})^2 \geq p_{\star} \left(\sum_i A_{ij}\right)^2$ again. Thus in every case~\eqref{eq: density in the columns} is satisfied by~$A$ (with~$p_{\star}$ instead of~$p$). A symmetric argument shows~$A$ satisfies~\eqref{eq: density in the rows} for~$p_{\star}$ as well. 

Thus $A \in \mathcal{S}(p_{\star})$ as claimed. However, since $p_{\star} > 4-2\sqrt{3}$, \cref{lemma: system has no solution} implies that $\mathcal{S}(p_{\star}) = \varnothing$, a contradiction. Thus the event $\mathcal{E}_{\mathrm{good}}$, which holds whp, does imply the event `$\mathrm{Left\  meets\  Right}$', proving the theorem.
\end{proof}

\section{Proof of \cref{theorem: 1-indep percolation in Z^2 times Kn,theorem: answer to Day FR Hancock question,theorem: pseudo random graphs,theorem: pseudo random graphs long paths}}\label{section: proofs of main theorems}
Our main theorems are all proved via a renormalisation argument combined with Theorem~\ref{theorem: threshold for left meets right}. Given two graphs $G$ and $H$, we may view the Cartesian product $H\times G$ as a kind of `augmented' version of $H$, and use any $1$-independent random graph $(\mathbf{H}\times \mathbf{G})_\mu$ on $H\times G$ to construct a new $1$-independent random graph $\mathbf{H}_{\nu}$ on $H$ as follows: given an edge $uv\in E(H)$, we let $uv$ be present in $
\mathbf{H}_{\nu}$ if  in the restriction of $(\mathbf{H}\times \mathbf{G})_\mu$ to $\{u,v\}\times V(G)$ there is a connected component containing strictly more than half of the vertices in each of $\{u\}\times V(G)$ and $\{v\}\times V(G)$.

That $\mathbf{H}_{\nu}$ is a $1$-independent random graph follows immediately from the fact that $(\mathbf{H}\times \mathbf{G})_\mu$ was $1$-independent: the states of edges inside vertex-disjoint edge-sets in $\mathbf{H}_{\nu}$ are determined by the states of edges inside vertex-disjoint edge sets in $(\mathbf{H}\times \textbf{G})_\mu$. Further, any path in $\mathbf{H}_{\nu}$ can be `lifted' up to a path in $(\mathbf{H}\times \mathbf{G})_\mu$ of equal or greater length: if $uv, vw$ are present in $\mathbf{H}_{\nu}$, then there exist connected subgraphs $C_{uv}$ and $C_{vw}$ in $(\mathbf{H}\times \mathbf{G})_\mu$ with $C_{uv}\subseteq \{u,v\}\times V(G)$, $C_{vw}\subseteq \{v,w\}\times V(G)$, $C_{uv}\cap \left(\{u\}\times V(G)\right)$ and $C_{vw}\cap \left(\{w\}\times V(G)\right)$ both non-empty, and $C_{uv}, C_{v,w}$ both containing strictly more than half of the vertices in $\{v\}\times V(G)$ (and hence having non-empty intersection).

Now the likelihood of an edge $uv$ being present in $\mathbf{H}_{\nu} $ is exactly the probability of the event corresponding to `Left meets Right' occurring in the restriction of $(\mathbf{H}\times \mathbf{G})_\mu$ to the vertex-set $\{u,v\}\times V(G)$ (which induces a copy of $K_2\times G$ in $H\times G$). Thus for $p>4-2\sqrt{3}$ and a suitable choice of $G$, we can use Theorem~\ref{theorem: threshold for left meets right}(i) to ensure that each edge in the $1$-independent random graph $\mathbf{H}_{\nu}$ is present with probability $1-o(1)$. With such a high edge probability, we can then establish the almost sure existence of infinite components or long paths in $\mathbf{H}_{\nu}$ in a straightforward way --- either by using results in the literature, or by a direct argument.

On the other hand if $p\leq 4-2\sqrt{3}$, we can use ideas from the lower bound construction in the proof of Theorem~\ref{theorem: threshold for left meets right}(ii), which date back to~\cite{DayFalgasRavryHancock20, FalgasRavry12}, in order to construct a $1$-independent random subgraph $\mathbf{G}$ of $H\times K_n$ that fails to percolate (or, if $H=\mathbb{Z}$, that only contain paths of length $O(n)$). For the convenience of the reader, we sketch below how this works in the special case $H=\mathbb{Z}^2$.

Take $p=4-2\sqrt{3}$, and set $\theta=(1+\sqrt{2p-1})/2$.  Independently assign to each vertex $(x,y,z)\in \mathbb{Z}^2\times V(K_n)$ a random state $S_{x,y,z} \in\{0,1,\star\}$ as follows:
\begin{itemize}
	\item if $\|(x,y)\|_{\infty}\cong 0 \mod6$, set $S_{x,y,z}=1$ with probability $1$;
	\item if $\|(x,y)\|_{\infty} \cong 1 \mod 6$, set $S_{x,y,z}=1$ with probability $\theta$, and $0$ otherwise;
	\item if $\|(x,y)\|_{\infty} \cong 2 \mod 6$, set $S_{x,y,z}=0$ with probability $\sqrt{p}$, and $\star$ otherwise;
	\item if $\|(x,y)\|_{\infty} \cong 3 \mod 6$, set $S_{x,y,z}=0$ with probability $1$;
	\item if $\|(x,y)\|_{\infty} \cong 4 \mod 6$, set $S_{x,y,z}=0$ with probability $\theta$, and $1$ otherwise;
	\item if $\|(x,y)\|_{\infty} \cong 5 \mod 6$, set $S_{x,y,z}=1$ with probability $\sqrt{p}$, and $\star$ otherwise.
\end{itemize}
We now use these random states to build a $1$-independent random graph $\mathbf{G}$ as follows. Given an edge$\{(x_1,y_1,z_1),(x_2,y_2,z_2)\}$ of $H\times K_n$, include it in $\mathbf{G}$ if one of the following holds:
\begin{itemize}
	\item $S_{x_1,y_1,z_1}=S_{x_2, y_2, z_2}\neq \star$
	 \item $\|(x_1,y_1)\|_{\infty} <\|(x_2,y_2)\|_{\infty}$ and $S_{x_2,y_2,z_2}=\star$. 
\end{itemize} 
Then the choice of probabilities for our random states ensure each edge is open with probability at least $p=4-2\sqrt{3}$, and our edge rules further imply that every connected component $C$ in $\mathbf{G}$ meets at most four consecutive  cylinders  $\mathcal{C}_r:=\{(x,y, z): \ \|(x,y)\|_{\infty}=r\}$, $r\in \mathbb{Z}_{\geq 0}$ since, as is easily checked, a connected component in $\mathbf{G}$ cannot both contain a vertex assigned state $0$ and a vertex assigned state $1$ --- we leave this as an exercise to the reader, and refer them to~\cite[Corollary 24]{DayFalgasRavryHancock20} for a proof of this fact in a more general setting. In particular, we have that $\mathbf{G}$ does not percolate.

Having thus outlined our proof ideas, we now fill in the details. First we formalise our renormalisation argument with the following lemma.
\begin{lemma}[Renormalisation lemma]\label{lemma: renormalisation}
Let~$H$ be a graph. Let $q=q(n)$ satisfy $nq(n)\gg \log n$, and let  $(G_n)_{n \in \mathbb{N}}$ be a sequence of $n$-vertex graphs which is weakly $q$-pseudorandom.
 Then for every $\eps>0$ and every $p>4-2\sqrt{3}$ fixed, there exists~$n_0$ such that for all $n\geq n_0$, $G=G_n$ and $\mu \in \mathcal{M}_{1, \geq p}(H\times G)$ there exists $\nu\in \mathcal{M}_{1, \geq 1-\eps}(H)$ and a coupling between $\mathbf{H}_{\nu}$ and $\mathbf{\left(H\times G\right)}_{\mu}$ such that there exists a path from~$u$ to~$v$ in $\mathbf{H}_{\nu}$ only if there exists a path from $\{u\}\times V(G)$ to $\{v\}\times V(G)$ in $\mathbf{\left(H\times G\right)}_{\mu}$.
\end{lemma}
\begin{proof}
Let $p>4-2\sqrt{3}$ and $\eps>0$ be fixed. By Theorem~\ref{theorem: threshold for left meets right}\cref{LMR_i}, there exists $n_0\in \mathbf{N}$ such that for all $n \geq n_0$ and all $\mu \in \mathcal{M}_{1,\geq p}(K_2\times G_n)$, the $\mu$-probability of the event `$\mathrm{Left\  meets\  Right}$' is at least~$1-\eps$.
For $n\geq n_0$, $G=G_n$ and $\mu \in \mathcal{M}_{1,\geq p}(H\times G)$, define a random graph model $\mathbf{H}_{\nu}$ from $\mathbf{\left(H\times G\right)}_{\mu}$ as follows: for each edge $uv\in E(H)$, we add~$uv$ to $\mathbf{H}_{\nu}$ if and only if there is a connected component in $\mathbf{\left(H\times G\right)}_{\mu}[\{u,v\}\times V(G_n)]$ containing strictly more than half of the vertices in $\{u\}\times V(G_n)$ and strictly more than half of the vertices $\{v\}\times V(G_n)$. 
The model $\mathbf{H}_{\nu}$ is clearly $1$-independent, has edge-probability at least~$1-\eps$, and has the property that any path in $\mathbf{H}_{\nu}$ can be lifted up to a path in $\mathbf{\left(H\times G\right)}_{\mu}$. This proves the Lemma.
\end{proof}
\noindent Recall that $2$-neighbour bootstrap percolation on a graph~$G$ is a discrete-time process defined as follows. At time $t=0$, an initial set of infected vertices $A=A_0$ is given. At every time $t\geq 1$, every vertex of~$G$ which has at least~$2$ neighbours in~$A_{t-1}$ becomes infected and is added to~$A_{t-1}$ to form~$A_t$. We denote by $\overline{A}$ the set of all vertices of~$G$ which are eventually infected, $\overline{A}=\bigcup_{t\geq 0}A_t$. Following Day, Falgas-Ravry and Hancock~\cite{DayFalgasRavryHancock20}, we say that a graph~$G$ has the \emph{finite $2$-percolation property} if for every finite set of initially infected vertices~$A$, the set of eventually infected vertices $\overline{A}$ is finite. The content of \cite{DayFalgasRavryHancock20}[Corollary 24] is, informally, that the construction based on random-states we outlined above `works on all host graphs that have the finite $2$-percolation property'.
\begin{proof}[Proof of Theorem~\ref{theorem: pseudo random graphs}]
Let $H=\mathbb{Z}^2$. Pick $\eps>0$ such that $1-\eps>0.8639$. Then by Lemma~\ref{lemma: renormalisation}, for any $p>4-2\sqrt{3}$,~$n$ sufficiently large and $G=G_n$, we can couple a random graph $\mathbf{\left(H\times G\right)}_{\mu}$, $\mu \in \mathcal{M}_{1,\geq p}(H)$ with a random graph $\mathbf{H}_{\nu}$, $\mu \in \mathcal{M}_{1,\geq 1-\eps}(H)$ such that if $\mathbf{H}_{\nu}$ percolates then so does $\mathbf{\left(H\times G\right)}_{\mu}$. Since $p_{1,c}(H)<0.86339$, as proved in~\cite[Theorem 2]{BalisterBollobasWalters05}, it follows that 
$p_{1,c}(H\times G)\leq p$. Since $p>4-2\sqrt{3}$ was arbitrary, we have the claimed upper bound $\lim_{n\rightarrow \infty}p_{1,c}(H\times G_n)\leq 4-2\sqrt{3}$. The lower bound $\lim_{n\rightarrow \infty}p_{1,c}(H\times G_n)\geq 4-2\sqrt{3}$ follows from~\cite[Corollary 24]{DayFalgasRavryHancock20} and the fact that $\mathbb{Z}^2\times G_n$ is easily seen to have the finite $2$-percolation property. Indeed, for any finite set of vertices~$A$ in $\mathbb{Z}^2\times G_n$, there is some finite~$N$ such that $A\subseteq [N]^2\times V(G_n)$. Now every vertex outside $[N]^2\times V(G_n)$ has at most one neighbour in $[N]^2\times V(G_n)$, and thus can never be infected by a $2$-neighbour bootstrap percolation process started from~$A$.
\end{proof}
\begin{remark}
The proof above in fact works in a more general setting than $\mathbb{Z}^2$: suppose~$H$ has the finite $2$-percolation property and satisfies $p_{1,c}(H)<1$. Let $(G_n)_{n \in \mathbb{N}}$ be a sequence of weakly $q$-pseudorandom $n$-vertex graphs with $nq(n)\gg \log n$. Then $H\times G_n$ also has the finite $2$-percolation property, and the proof above shows
\[\lim_{n\rightarrow \infty}p_{1,c}(H\times G_n)=4-2\sqrt{3}.\]
Examples of graphs with the finite $2$-percolation property include many of the standard lattices studied in percolation theory, such as the honeycomb (hexagonal) lattice, the dice (rhombile) lattice or the tetrakis (`Union Jack') lattice.
\end{remark}
\begin{proof}[Proof of Theorem~\ref{theorem: 1-indep percolation in Z^2 times Kn}]
	Since~$K_n$ is $1$-pseudorandom, Theorem~\ref{theorem: 1-indep percolation in Z^2 times Kn} is immediate from Theorem~\ref{theorem: pseudo random graphs}.
\end{proof}
\begin{proof}[Proof of Theorem~\ref{theorem: pseudo random graphs long paths}]
	Let $H=\mathbb{Z}^2$. Pick $\eps>0$ such that $1-\eps>3/4$. Then by Lemma~\ref{lemma: renormalisation}, for any $p>4-2\sqrt{3}$,~$n$ sufficiently large and $G=G_n$, we can couple a random graph $\mathbf{\left(H\times G\right)}_{\mu}$, $\mu \in \mathcal{M}_{1,\geq p}(H)$ with a random graph $\mathbf{H}_{\nu}$, $\mu \in \mathcal{M}_{1,\geq 1-\eps}(H)$ such that if $\mathbf{H}_{\nu}$ contains a path of length~$\ell$ then so does $\mathbf{\left(H\times G\right)}_{\mu}$. Since $p_{1,LP}(H)=\frac{3}{4}$, as proved in~\cite[Theorem 11(i)]{DayFalgasRavryHancock20}\footnote{For the proof of this theorem, all we need is $p_{1,LP}(H)<1$, and thus the weaker bound $p_{1,LP}(H)\leq 1-1/3e$ (which follows directly from an application of the Lov\'asz local lemma) would suffice for our purposes here.} it follows that 
	$p_{1,LP}(H\times G)\leq p$. Since $p>4-2\sqrt{3}$ was arbitrary, we have the claimed upper bound $\lim_{n\rightarrow \infty}p_{1,LP}(H\times G_n)\leq 4-2\sqrt{3}$. The lower bound $\lim_{n\rightarrow \infty}p_{1,c}(H\times G_n)\geq 4-2\sqrt{3}$ was proved in~\cite[Theorem 12(v)]{DayFalgasRavryHancock20} (with the same construction as we outlined at the beginning of this section, adapted mutatis mutandis to the setting $H=\mathbb{Z}$).
\end{proof}
\begin{proof}[Proof of Theorem~\ref{theorem: answer to Day FR Hancock question}]
		Since~$K_n$ is $1$-pseudorandom, Theorem~\ref{theorem: answer to Day FR Hancock question} is immediate from Theorem~\ref{theorem: pseudo random graphs long paths}.
\end{proof}
\section{Component evolution in $1$-independent models}\label{section: component evolution}
\noindent Recall that the independence number~$\alpha(G)$ of a graph~$G$ is the size of a largest independent (edge-free) subset of~$V(G)$, and that a \emph{perfect matching} in a graph~$G$ is a matching whose edges together cover all the vertices in~$V(G)$. Moreover, a graph $G$ is a \emph{complete multipartite graph} if there exists a partition of $V(G)$ such that two vertices in $V(G)$ are joined by an edge in $G$ if and only if they are contained in different parts of the partition. Finally, the \emph{complement} $G^c$ of a graph $G$ is the graph on $V(G)$ whose edges are the non-edges of $G$, $G^c:=(V(G), V(G)^{(2)}\setminus E(G))$.
\begin{lemma}\label{lemma: minimising the number of PM}
	If~$G$ is a complete multipartite graph on~$2n$ vertices with independence number $\alpha(G)\leq n$, then~$G$ contains at least~$n!$ perfect matchings.
\end{lemma}	
\begin{proof}
	Let~$G$ be a complete multipartite graph on~$2n$ vertices with the minimum number of perfect matchings subject to $\alpha(G)\leq n$. Let $V_1, V_2, \dots, V_r$ denote the parts of~$G$ with $\s{V_1} \geq \s{V_2} \geq \dots \geq \s{V_r}$.
	If $\s{V_{r-1}} + \s{V_r} \leq n$, then the graph~$G'$ obtained from~$G$ by deleting all edges in $G[V_{r-1}, V_r]$ satisfies $\alpha(G')\leq n$ and has at most as many perfect matchings as~$G$. We may therefore assume that $\s{V_{r-1}} + \s{V_r} \geq n$, and thus in particular that $r \leq 3$. Consider a perfect matching~$M$ in~$G$ and let~$i$ be the number of edges in $E(G[V_1, V_2]) \cap M$. Clearly $\s{E(G[V_1, V_3]) \cap M} = \s{V_1} - i$ and $\s{E(G[V_2, V_3])\cap M} = \s{V_2} - i = \s{V_3} - (\s{V_1} -i)$. From this we deduce that $i = \frac{1}{2}(\s{V_1} + \s{V_2} - \s{V_3}) = n - \s{V_3}$. Hence the number $\mathrm{PM}(G) $ of perfect matchings in~$G$ is:
	\begin{align*}
	\mathrm{PM}(G) &= \binom{\s{V_1}}{i} \binom{\s{V_2}}{i} \binom{\s{V_3}}{\s{V_1}-i} i! (\s{V_2} - i)! (\s{V_1} - i)! = \frac{\s{V_1}!\s{V_2}!\s{V_3}!}{(n - \s{V_1})! (n - \s{V_2})! (n - \s{V_3})!}.
	\end{align*}
	(Here $\binom{\s{V_1}}{i} \binom{\s{V_2}}{i}i!$ counts the number of different ways of selecting $i$-sets of vertices from each of $V_1$ and $V_2$ and joining them by a perfect matching, while $\binom{\vert V_3\vert}{\vert V_1\vert-i} (\s{V_2} - i)! (\s{V_1} - i)!$ counts the number of ways of joining the vertices of $V_3$ by a perfect matching to the remaining vertices of $V_1\cup V_2$.)

	\noindent If $\s{V_3} > 0$, then let~$G'$ be the complete tripartite graph with parts of size $\s{V_1}, \s{V_2}+1, \s{V_3}-1$. Note that $\alpha(G') \leq n$. By the formula above , we have 
	\[
	\frac{\mathrm{PM}(G)}{\mathrm{PM}(G')} = \frac{\s{V_3}(n-\s{V_3}+1)}{(\s{V_2}+1)(n-\s{V_2})} \geq 1,
	\]
	since $\s{V_3}(n-\s{V_3}+1) - (\s{V_2}+1)(n-\s{V_2}) = (\s{V_2}-\s{V_3}+1)(\s{V_2} + \s{V_3} - n) \geq 0$ (as $\s{V_2} \geq \s{V_3}$ and $\s{V_2} + \s{V_3} \geq n$). It follows that $\mathrm{PM}(G) \geq \mathrm{PM}(K_{n,n}) = n!$ as claimed.
\end{proof}
\begin{proof}[Proof of Proposition~\ref{prop: two-state measure minimises prob of compomnent of size 1/2}]
Let $H=K_{2n}$.
For all $p\in [\frac{1}{2},1]$, we may construct the two-state measure $\mu_{2s,p} \in \M_{1,p}(H)$ which satisfies:
	\begin{align*}
	\mathbb{P}\left[ \vert C_1(\mathbf{H}_{\mu_{2s}, p})\vert \leq n \right]=	\mathbb{P}\left[ \vert C_1(\mathbf{H}_{\mu_{2s}, p})\vert = n \right] =\binom{2n}{n}\theta^n(1-\theta)^n=\binom{2n}{n}\left(\frac{1-p}{2}\right)^n,
	\end{align*}
	proving the upper bound in that range. For $p_{2n}\leq p\leq \frac{1}{2}$, we note that $\theta=\theta(p)$ is no longer a real number. However, as shown in~\cite[Section 7.1]{DayFalgasRavryHancock20}, we may take a `complex limit' of the $2$-state measure~$\mu_{2s,p}$, and the conclusion above still holds.

	For the lower bound, let $C_1, C_2, \ldots ,C_r$ be the connected components of a $\mu$-random subgraph $\mathbf{H}_{\mu}$ of~$K_{2n}$. Let~$\mathbf{G}$ denote 
	the complete multipartite graph associated with the partition $\sqcup_{i}C_i$ of $V(K_{2n})=[2n]$. Observe that $\mathbf{G}$ is a subgraph of  the complement $\mathbf{H}_{\mu}^c$ of $\Hm$. If $\vert C_i \vert \leq n$ for all~$i$, then $\alpha(\mathbf{G})\leq n$, whence by Lemma~\ref{lemma: minimising the number of PM}~$\mathbf{G}$ contains at least~$n!$ perfect matchings. In particular, $\mathbf{H}_{\mu}^c$ must contain at least~$n!$ perfect matchings. By Markov's inequality, we thus have
	\begin{align*}
	\mathbb{P}\left[\vert C_1(\Hm)\vert \leq n \right]&\leq 	\mathbb{P}\left[\Hm^c \textrm{ contains } \geq n! \textrm{ perfect matchings}\right]\\
	& \leq \frac{1}{n!}\mathbb{E}\left[\#\{\textrm{perfect matchings in }\Hm^c\}\right]
	\\
	&= \frac{1}{n!}\left(\frac{1}{n!} \prod_{i=0}^{n-1}\binom{2n-2i}{2}\right)(1-p)^n=\binom{2n}{n}\left(\frac{1-p}{2}\right)^n.
	\end{align*}
	(Here $\left(\frac{1}{n!} \prod_{i=0}^{n-1}\binom{2n-2i}{2}\right)$ counts the number of perfect matchings in $K_{2n}$ by selecting $n$ vertex-disjoint edges sequentially one after the other, and dividing through by $n!$.) The lower bound follows.
\end{proof}
\begin{proof}[Proof of Theorem~\ref{theorem: component size pseudo-random graph}]
Let $p\in (\frac{1}{r+1}, \frac{1}{r}]$ be fixed. Fix $\eps=\eps(p)>0$ sufficiently small. For~$n$ large enough, we have by the pseudorandomness assumption on~$H_n$ that for every $U\subseteq V(H_n)$, $e(H_n[U])\leq q\frac{\s{U}^2}{2}+\eps^2 pqn^2$. It then follows from Lemma~\ref{lem:edge_concentration} that whp 
\begin{align}\label{eq: bound on numer of edges in Hm}
e(\Hm)\geq pq\frac{n^2}{2}(1-4\eps^2),
\end{align}
which is strictly greater than $\frac{qn^2}{2(r+1)}$ for $\eps=\eps(p)$ chosen sufficiently small. Assume~\eqref{eq: bound on numer of edges in Hm}. We show this implies the claimed lower bound on the size of a largest component.

If $\vert C_1(\Hm)\vert \leq \frac{n}{r+1}-\varepsilon n$, then for~$\eps$ sufficiently small there is a partition of~$V(H)$ into at most~$2(r+1)+1$ sets, each of which has size at most $\frac{n}{r+1}-\varepsilon n$, such that every connected component of~$\Hm$ is wholly contained in one of the sets of the partition. Indeed, such a partition can be obtained by starting with a partition of $V(H)$ into the connected components of~$\Hm$, and then as long as the partition contains two parts of size at most $\frac{1}{2}\left(\frac{n}{r+1}-\eps n\right)$, choosing two such parts arbitrarily and merging them into a single part. Since for any $(2r+3)$-tuple $(x_1, \ldots, x_{2r+3})$ with $\frac{1}{r+1}-\varepsilon\geq x_i \geq 0$ and $\sum_i x_i=1$ we have $\sum_i (x_i)^2\leq (r+1)\left(\frac{1}{r+1}-\varepsilon\right)^2 + \left((r+1)\eps\right)^2$, we have by our pseudorandomness assumption that 
\begin{align*}
e(\mathbf{H}_{\mu})\leq \frac{q(r+1)}{2}\left(\frac{1}{r+1}-\varepsilon\right)^2n^2+ \frac{q}{2}\left((r+1)\eps\right)^2n^2 + (2r+3)\eps^2pqn^2<\frac{qn^2}{2(r+1)}
\end{align*}
 for~$\eps$ sufficiently small, contradicting~\eqref{eq: bound on numer of edges in Hm}. Thus we may assume that $\vert C_1(\Hm)\vert>\frac{n}{r+1}-\varepsilon n$.

 If $\vert C_1(\Hm)\vert \geq \frac{n}{r}$, then we have nothing to show. Finally if $\frac{n}{r+1}-\varepsilon n\leq \vert C_1(\Hm)\vert < \frac{n}{r}$, then~$\Hm$ contains at least~$r+1$ components. Let $\alpha n$ denote the size of a largest component, where  $\frac{1}{r+1}-\eps < \alpha < \frac{1}{r}$. Then
	\[\left(r\alpha^2 + (1-r\alpha)^2 \right)q\frac{n^2}{2} +(r+2)\eps^2pqn^2\geq e(\mathbf{H}_{\mu})\geq pq\frac{n^2}{2}(1-4\eps^2).\] 
	Dividing through by~$qn^2/2$, rearranging terms and using the fact~$\eps$ is chosen sufficiently small, we get
\[ r\alpha^2 +(1-r\alpha)^2 \geq p -\eps.\]
	Solving for~$\alpha$, we get that
	\[\alpha \geq  \frac{1+\sqrt{\frac{(r+1)(p-\eps)-1}{r}}}{r+1},\]
	giving part (i).

	For part (ii), consider the $r+1$-state measure in which each vertex is assigned state~$r+1$ with probability $\frac{1-\sqrt{r\left((r+1)p-1\right)}}{r+1}$ and a uniform random state from the set $\{1,2, \ldots, r\}$ otherwise, and in which an edge is open if and only if its vertices are in the same state. This is easily seen to be a $1$-ipm with the requisite properties.
\end{proof}

\subsection*{Acknowledgements}
The authors would like to thank two anonymous referees for their careful work, which helped improve the exposition in the paper. This research was carried out while the second author visited Ume{\aa} University under the auspices of an Erasmus exchange scheme, whose support is gratefully acknowledged.


\begin{thebibliography}{10}
	
	\bibitem{BalintBeffaraTassion13}
	Andr{\'a}s B{\'a}lint, Vincent Beffara, and Vincent Tassion.
	\newblock On the critical value function in the divide and color model.
	\newblock {\em Latin American Journal of Probability and Mathematical
		Statistics}, 10(2):653--666, 2013.
	
	\bibitem{BalisterBollobas12}
	Paul Balister and B{\'e}la Bollob{\'a}s.
	\newblock Critical probabilities of 1-independent percolation models.
	\newblock {\em Combinatorics, Probability and Computing}, 21(1-2):11--22, 2012.
	
	\bibitem{BalisterBollobas13}
	Paul Balister and B{\'e}la Bollob{\'a}s.
	\newblock Percolation in the k-nearest neighbor graph.
	\newblock In {\em Recent Results in Designs and Graphs: a Tribute to Lucia
		Gionfriddo}, volume~28 of {\em Quaderni di Matematica}, pages 83--100. 2013.
	
	\bibitem{BalisterBollobasSarkarKumar07}
	Paul Balister, B{\'e}la Bollobas, Amites Sarkar, and Santosh Kumar.
	\newblock Reliable density estimates for coverage and connectivity in thin
	strips of finite length.
	\newblock In {\em Proceedings of the 13th annual ACM international conference
		on Mobile computing and networking}, pages 75--86. ACM, 2007.
	
	\bibitem{BalisterBollobasWalters05}
	Paul Balister, B{\'e}la Bollob{\'a}s, and Mark Walters.
	\newblock Continuum percolation with steps in the square or the disc.
	\newblock {\em Random Structures \& Algorithms}, 26(4):392--403, 2005.
	
	\bibitem{BalisterBollobasWalters09}
	Paul Balister, B{\'e}la Bollob{\'a}s, and Mark Walters.
	\newblock Random transceiver networks.
	\newblock {\em Advances in Applied Probability}, 41(2):323--343, 2009.
	
	\bibitem{Ball14}
	Neville Ball.
	\newblock Rigorous confidence intervals on critical thresholds in 3 dimensions.
	\newblock {\em Journal of Statistical Physics}, 156(3):574--585, 2014.
	
	\bibitem{BenjaminiStauffer13}
	Itai Benjamini and Alexandre Stauffer.
	\newblock Perturbing the hexagonal circle packing: a percolation perspective.
	\newblock In {\em Annales de l'Institut Henri Poincar{\'e}, Probabilit{\'e}s et
		Statistiques}, volume~49, pages 1141--1157. Institut Henri Poincar{\'e},
	2013.
	
	\bibitem{BollobasRiordan06}
	B{\'e}la Bollob{\'a}s and Oliver Riordan.
	\newblock {\em Percolation}.
	\newblock Cambridge University Press, 2006.
	
	\bibitem{DayFalgasRavryHancock20}
	A~Nicholas Day, Victor Falgas-Ravry, and Robert Hancock.
	\newblock Long paths and connectivity in 1-independent random graphs.
	\newblock {\em Random Structures \& Algorithms}, 57(4):1007--1049, 2020.
	
	\bibitem{DeanKouider00}
	Nathaniel Dean and Mekkia Kouider.
	\newblock Gallai's conjecture for disconnected graphs.
	\newblock {\em Discrete Mathematics}, 213(1-3):43--54, 2000.
	\newblock Selected topics in discrete mathematics (Warsaw, 1996).
	
	\bibitem{DeijfenHaggstromHolroyd12}
	Maria Deijfen, Olle H{\"a}ggstr{\"o}m, and Alexander~E. Holroyd.
	\newblock Percolation in invariant {P}oisson graphs with iid degrees.
	\newblock {\em Arkiv f{\"o}r Matematik}, 50(1):41--58, 2012.
	
	\bibitem{DeijfenHolroydPeres11}
	Maria Deijfen, Alexander~E. Holroyd, and Yuval Peres.
	\newblock Stable poisson graphs in one dimension.
	\newblock {\em Electronic Journal of Probability}, 16:1238--1253, 2011.
	
	\bibitem{FalgasRavry12}
	Victor Falgas-Ravry.
	\newblock {\em Thresholds in probabilistic and extremal combinatorics.}
	\newblock PhD thesis, University of London, 2012.
	
	\bibitem{Grimmett99}
	Geoffrey~R. Grimmett.
	\newblock {\em Percolation}, volume 321.
	\newblock Springer, 1999.
	
	\bibitem{HaenggiSarkar13}
	Martin Haenggi and Amites Sarkar.
	\newblock Percolation in the secrecy graph.
	\newblock {\em Discrete Applied Mathematics}, 161(13-14):2120--2132, 2013.
	
	\bibitem{Harris60}
	Theodore~E. Harris.
	\newblock A lower bound for the critical probability in a certain percolation
	process.
	\newblock {\em Mathematical Proceedings of the Cambridge Philosophical
		Society}, 56(01):13--20, 1960.
	
	\bibitem{Kesten80}
	Harry Kesten.
	\newblock The critical probability of bond percolation on the square lattice
	equals 1/2.
	\newblock {\em Communications in Mathematical Physics}, 74(1):41--59, 1980.
	
	\bibitem{KrivelevichSudakov06}
	Michael Krivelevich and Benny Sudakov.
	\newblock Pseudo-random graphs.
	\newblock In {\em More sets, graphs and numbers}, pages 199--262. Springer,
	2006.
	
	\bibitem{LiggettSchonmannStacey97}
	Thomas~M. Liggett, Roberto~H. Schonmann, and Alan~M. Stacey.
	\newblock Domination by product measures.
	\newblock {\em The Annals of Probability}, 25(1):71--95, 1997.
	
	\bibitem{Lyons90}
	Russell Lyons.
	\newblock Random walks and percolation on trees.
	\newblock {\em The Annals of Probability}, pages 931--958, 1990.
	
	\bibitem{MeesterRoy96}
	Ronald Meester and Rahul Roy.
	\newblock {\em Continuum percolation}, volume 119.
	\newblock Cambridge University Press, 1996.
	
	\bibitem{Meester94}
	Ronald~W.J. Meester.
	\newblock Uniqueness in percolation theory.
	\newblock {\em Statistica Neerlandica}, 48(3):237--252, 1994.
	
	\bibitem{RiordanWalters07}
	Oliver Riordan and Mark Walters.
	\newblock Rigorous confidence intervals for critical probabilities.
	\newblock {\em Physical Review E}, 76(1):011110, 2007.
	
	\bibitem{Thomason87}
	Andrew Thomason.
	\newblock Pseudo-random graphs.
	\newblock In M.~Karo\'nski, editor, {\em Proceedings of Random Graphs, Pozna\'n
		1985}, volume~33 of {\em Annals of Discrete Mathematics}, pages 307--331.
	North-Holland, 1987.
	
	\bibitem{VDBErmakov96}
	Jacob van~den Berg and Andrei Ermakov.
	\newblock A new lower bound for the critical probability of site percolation on
	the square lattice.
	\newblock {\em Random Structures \& Algorithms}, 8(3):199--212, 1996.
	
	\bibitem{Wierman95}
	John~C. Wierman.
	\newblock Substitution method critical probability bounds for the square
	lattice site percolation model.
	\newblock {\em Combinatorics, Probability and Computing}, 4:181--188, 1995.
	
	\bibitem{Ziff92}
	Robert~M Ziff.
	\newblock Spanning probability in 2d percolation.
	\newblock {\em Physical Review Letters}, 69(18):2670, 1992.
	
\end{thebibliography}
\end{document}